\documentclass[12pt, leqno, a4paper]{amsart}
\usepackage{amsmath,amsthm,amscd,amssymb,amsfonts, amsbsy}
\usepackage{latexsym}
\usepackage{exscale}

\newcommand{\M}{\mathcal{M}}

\usepackage[colorlinks=true, pdfstartview=FitV, linkcolor=blue, citecolor=red, urlcolor=blue, hypertexnames=false]{hyperref}

\newcommand{\Bs}{\mathcal{B}}
\newcommand{\R}{\mathbb{R}}

\newcommand{\re}{\mathbb{R}}

\newcommand{\dd}{\mathcal{D}}

\newcommand{\loc}{{\rm loc}}

\newcommand{\bmo}{{\rm BMO}}

\newcommand{\dya}{{\rm dyadic}}

\def\mvint_#1{\mathchoice%
          {\mathop{\kern 0.2em\vrule width 0.6em height 0.69678ex depth -0.58065ex
                  \kern -0.8em \intop}\nolimits_{\kern -0.4em#1}}%
          {\mathop{\kern 0.1em\vrule width 0.5em height 0.69678ex depth -0.60387ex
                  \kern -0.6em \intop}\nolimits_{#1}}%
          {\mathop{\kern 0.1em\vrule width 0.5em height 0.69678ex depth -0.60387ex
                  \kern -0.6em \intop}\nolimits_{#1}}%
          {\mathop{\kern 0.1em\vrule width 0.5em height 0.69678ex depth -0.60387ex
                  \kern -0.6em \intop}\nolimits_{#1}}}

\def\aver_#1{\mvint_{#1}}

\theoremstyle{plain}
\newtheorem{theorem}[equation]{Theorem}
\newtheorem{lemma}[equation]{Lemma}
\newtheorem{corollary}[equation]{Corollary}

\theoremstyle{definition}
\newtheorem{definition}[equation]{Definition}
\newtheorem{example}{Example}

\theoremstyle{remark}
\newtheorem{remark}[equation]{Remark}

\calclayout

\numberwithin{equation}{section}

\begin{document}

\allowdisplaybreaks

\title[Oscillation, self-improving and  good-$\lambda$ inequalities]{Oscillation estimates, self-improving results and good-$\lambda$ inequalities}

\author{Lauri Berkovits}

\address{Lauri Berkovits
\\
University of Oulu
\\
Department of Mathematical Sciences
\\
P.O. Box 3000
\\FI-90014 Oulu, Finland}

\email{lauri.berkovits@oulu.fi}

\author{Juha Kinnunen}

\address{Juha Kinnunen
\\
Aalto University
\\
Department of Mathematics
\\
P.O. Box 11100
\\FI-00076 Aalto, Finland}
\email{juha.k.kinnunen@aalto.fi}

\author{Jos\'e Mar{\'\i}a Martell}

\address{Jos\'e Mar{\'\i}a Martell
\\
Instituto de Ciencias Matem\'aticas CSIC-UAM-UC3M-UCM
\\
Consejo Superior de Investigaciones Cient{\'\i}ficas
\\
C/ Nicol\'as Cabrera, 13-15
\\
E-28049 Madrid, Spain} \email{chema.martell@icmat.es}

\thanks{The first author is grateful for the hospitality of ICMAT in Madrid, Spain, where
parts of this research was conducted. The visit was supported by the
Finnish Academy of Science and Letters, Vilho, Yrj\"o and Kalle V\"ais\"al\"a
foundation. The second author was supported by the Academy of Finland.
The third author was supported by ICMAT Severo Ochoa project SEV-2011-0087. He also acknowledges that the research leading to these results has received funding from the European Research Council under the European Union's Seventh Framework Programme (FP7/2007-2013)/ ERC agreement no. 615112 HAPDEGMT}

\date{July 9, 2015. \textit{Revised}\textup{:} \today}



\begin{abstract}
Our main result is an abstract good-$\lambda$ inequality that allows us to consider three self-improving properties related to oscillation estimates in a very general context. The novelty of our approach is that there is one principle behind these self-improving phenomena. First, we obtain higher integrability properties for functions belonging to the so-called John-Nirenberg spaces. Second, and as a consequence of the previous fact, we present very easy proofs of some of the self-improving properties of the generalized Poincar\'e inequalities studied by  B. Franchi, C. P\'erez and R. Wheeden in \cite{FPW}, and by P. MacManus and C. P\'erez in \cite{MaPe}. Finally, we show that a weak Gurov-Reshetnyak condition implies higher integrability with asymptotically sharp estimates.  We discuss these questions both in Euclidean spaces with dyadic cubes and in spaces of homogeneous type with metric balls. We develop new techniques that apply to more general oscillations than the standard mean oscillation and to overlapping balls instead of dyadic cubes.
\end{abstract}

\maketitle

{\parskip=0cm \small\tableofcontents}

\section{Motivation: self-improving phenomena}

It is well-known that the Sobolev-Poincar\'e inequality
\begin{equation}\label{Poincare:1,p}
\aver_Q|f(x)-f_Q|\,dx
\lesssim
\ell(Q)\left(\aver_{Q}|\nabla f(x)|^p\,dx\right)^{1/p}
\end{equation}
encodes a self-improvement in the local integrability of $f$. Indeed, the
previous estimate is meaningful provided $f\in L^1_\loc(\re^n)$
and $\nabla f\in L^p_\loc(\re^n)$ and it implies
\begin{equation}\label{Poincare:p*,p}
\left(\aver_Q|f(x)-f_Q|^{p^*}\,dx\right)^{1/p^*}
\lesssim
\ell(Q)\left(\aver_{Q}|\nabla f(x)|^p\,dx\right)^{1/p}
\end{equation}
with $p^*=pn/(n-p)$ for $1 \leq p < n$. If $p\ge n$ we obtain a similar estimate
for any $p^* \in (1,\infty)$. Here $f_Q$ and the barred integral sign both denote
the integral average and $\ell(Q)$ stands for the side length of a cube $Q$.
Denoting the right-hand side of \eqref{Poincare:1,p} by $a(Q)$, the
inequality may be rewritten as
\begin{equation}\label{Poincare:gener}
\aver_Q|f(x)-f_Q|\,dx
\le
a(Q).
\end{equation}
In general, we may study generalized Poincar\'e inequalities of the
form \eqref{Poincare:gener} with respect to an abstract functional
$a$ acting on cubes. The inequality
\eqref{Poincare:1,p} above is one of the most relevant examples, but
inequalities involving controlled oscillation appear frequently both in the Euclidean and non-Euclidean setting.
For instance, the
Sobolev-Poincar\'e inequality has an analogue in metric measure spaces
(defined in terms of the so-called upper gradients) which
has become a standard tool in the field, see \cite{Hein01}.

A unified approach to the subject
was first developed in \cite{FPW}, in the context of spaces of homogeneous type.
They introduced a discrete summability condition $D_p$, which in the dyadic
setting takes the following form. Given a cube $Q_0$, an exponent $p$ with $1<p<\infty$, and a
functional $a:\dd(Q_0)\rightarrow [0,\infty)$  ---here and elsewhere we will write $\dd(Q_0)$ to denote the family of
dyadic subcubes of $Q_0$--- we say that $a\in D_p^\dya(Q_0)$,
if there exists a constant $\|a\|$ such that for every $Q\in\dd(Q_0)$, we have
$$
\sum_{i} a(Q_i)^p|Q_i|
\le
\|a\|^p a(Q)^p|Q|,
$$
whenever $\{Q_i\}_i\subset \dd(Q)$ is a pairwise disjoint family.
It was shown in \cite{FPW} that if
\eqref{Poincare:gener} holds for all $Q \in \dd(Q_0)$ with
$a\in D_p^\dya(Q_0)$, then
\begin{equation}\label{Poincare:gener_self-imp:weak}
\|f-f_Q\|_{L^{p,\infty}, Q}
\lesssim \|a\|a(Q)
\end{equation}
for every $Q\in\dd(Q_0)$.
In the previous expression we have used the following notation:
given a Banach function space $\mathbb{X}$ (e.g., $L^p$, $L^{p,\infty}$, etc.)
and a cube $Q$, we write $\|f\|_{X,Q}=\|f\|_{X(Q,dx/Q)}$. Note
that \eqref{Poincare:gener_self-imp:weak} and Kolmogorov's inequality
imply
\begin{equation}\label{Poincare:gener_self-imp}
\left(\aver_Q|f(x)-f_Q|^q\,dx\right)^{1/q}
\lesssim \|a\|a(Q)
\end{equation}
for every $1\le q<p$.
Thus,  we have again a self-improvement phenomenon: a priori we only
have $f\in L^1_\loc(\re)$ and a posteriori we get $f\in L^q_\loc(\re)$
for every $1 \leq q<p$. The results in \cite{FPW} were extended
and improved in \cite{LePe, MaPe} and we will further generalize them.

Another, apparently different, self-improvement takes
place for the functions belonging to the John-Nirenberg
spaces which are defined as follows. Given $f\in L^1(Q_0)$ and $1<p<\infty$, we say that $f\in JN_p^\dya(Q_0)$
provided
\begin{align*}
\|f\|_{JN_p^\dya(Q_0)}
:&
=
\sup_{Q\in\dd(Q_0)} \|f\|_{JN_p^\dya,Q}<\infty,
\end{align*}
where
\begin{align*}
\|f\|_{JN_p^\dya,Q}
:&=
\sup \left(\frac1{|Q|}\sum_i \left(\aver_{Q_i} |f(x)-f_{Q_i}|\,dx\right)^{p}|Q_i|\right)^{1/p},
\end{align*}
and  the supremum is taken over all pairwise disjoint subfamilies $\{Q_i\}_i$ of $\dd(Q)$.
These spaces first appeared in the celebrated paper of F. John and
L. Nirenberg \cite{JoNi} and the space $\bmo$ can be seen as the
limit case of $JN_p$ as $p \to \infty$, see also \cite{Camp,Gius,GiMa}.
It was shown in \cite{JoNi} that
the space $JN_p(Q_0)$ embeds into $L^{p,\infty}(Q_0)$, which again amounts to improvement in the order of integrability of $f$.
We shall show that $JN_p$ spaces and generalized Poincar\'e inequalities
are closely connected. In particular, the embedding
$JN_p(Q_0) \hookrightarrow L^{p,\infty}(Q_0)$ easily implies
some of the known self-improvement results for generalized Poincar\'e
inequalities, including \eqref{Poincare:gener_self-imp:weak}.

The last example of self-improvement that we consider is given by the Gurov-Reshet\-nyak condition, first introduced in the context of
quasiconformal mappings, see \cite{Guro1, GuRe, Re2}.
For a non-negative function $w\in L^1(Q_0)$ (called a weight), we write
$w\in GR_\varepsilon^\dya(Q_0)$, where $0<\varepsilon<2$, if
\begin{equation}\label{GR-intro-eq}
\aver_Q|w(x)-w_Q|\,dx
\le
\varepsilon w_Q
\end{equation}
for every $Q\in\dd(Q_0)$.
This condition implies that $w\in L^{p_{\varepsilon}}(Q_0)$ for some $p_{\varepsilon}>1$,
see \cite{AB, Boja1,Boja2, Isan1, Isan2, Iw1, KoLeSt, Re1, Re2}.
The main point of interest here is that  $p_{\varepsilon}\to +\infty$ as  $\varepsilon \to 0^+$.
 While \eqref{GR-intro-eq} is of the form \eqref{Poincare:gener}, the results
in \cite{FPW,LePe,MaPe} do not provide any non-trivial information about the class
$GR_\varepsilon^\dya(Q_0)$. This is because $a(Q)=\varepsilon w_Q$
 only satisfies $D_p$ with $p=1$ (see \cite[p. 3]{MaPe}).
However, our approach applies to Gurov-Reshetnyak weights as well.

The novelty of our approach is to show that there is one principle behind these
self-improvement phenomena: they all
(and much more) can be derived from a single abstract good-$\lambda$ inequality, which is a refined local version of the two-parameter good-$\lambda$ inequalities considered \cite{AM}.

In Part \ref{part:I} we consider the dyadic (and local) case, and the related good-$\lambda$ inequality is contained in  Theorem \ref{theor:good-lambda:dyadic}.
Our first set of applications (see Section \ref{section:App-I}) contains the examples of self-improving estimates pointed out above. We first obtain an embedding of the John-Nirenberg space into the corresponding weak Lebesgue space. Second we show how this embedding easily gives some of the
Franchi-P\'erez-Wheeden self-improvements in \cite{FPW, MaPe}. Finally, we frame the Gurov-Reshet\-nyak condition into our good-$\lambda$ inequality to obtain the asymptotic higher integrability. We would like to emphasize that these applications are straightforward once the good-$\lambda$ result is available.

Another important feature of our good-$\lambda$ inequality is that we can consider different oscillations, that is, $|f(x)-f_Q|$ may be replaced by
$|f(x)-A_Q f(x)|$, where $A_Q$ is a local operator. In Section \ref{section:App-II} we elaborate on this and obtain
self-improvements for new John-Nirenberg, Franchi-P\'erez-Wheeden and Gurov-Reshetnyak conditions written in terms of these local oscillations.

In Part \ref{part:II} we consider the corresponding problems but in the setting of spaces of homogeneous type, that is, in metric spaces endowed with a doubling measure. We obtain a local good-$\lambda$ inequality (see Section \ref{section:good-lambda:X}), which is applied to the self-improving properties. We consider more general John-Nirenberg, Franchi-P\'erez-Wheeden and Gurov-Reshetnyak conditions which are natural when working with the metric balls. We would like to emphasize that in contrast with Part \ref{part:I}, where cubes can be nicely decomposed as a union of non-overlapping cubes, in Part \ref{part:II}, coverings are made with balls. This creates both overlap and ``increases the support'' (that is, instead of working in a given ball $B$ we have to consider the dilated ball $(1+\gamma)\,B$).

Good-$\lambda$ inequalities typically lead to weighted and unweighted estimates. In this paper we will only consider unweighted estimates for the sake of conciseness. The corresponding weighted norm inequalities with Muckenhoupt weights will be treated elsewhere.

\part{The Euclidean setting: dyadic cubes}\label{part:I}

\section{The good-$\lambda$ inequality}
The main result in this section is an abstract local good-$\lambda$ inequality written in terms of dyadic cubes. To set the stage, we fix  a cube $Q_0\subset\re^n$. We recall that $\dd(Q_0)$ stands for the set of dyadic subcubes of $Q_0$. If $Q \in \dd(Q_0)\setminus\{Q_0\}$ we write $\widehat{Q}$ for the dyadic parent of $Q$, that is, the unique $\widehat{Q}\in\dd(Q_0)$ with  side length $\ell(\widehat{Q})=2\ell(Q)$. Let $\M_{Q_0}$ denote the local dyadic maximal operator
$$
\M_{Q_0} f(x)
=
\sup_{x\in Q\in\dd(Q_0)} \aver_{Q} |f(y)|\,dy.
$$

We are now ready to state our good-$\lambda$ inequality which is a refined local version of the two-parameter good-$\lambda$ inequalities considered in \cite{AM}. The proof is postponed until Section \ref{section:proof:good-lambda-Q}.

\begin{theorem}\label{theor:good-lambda:dyadic}
Fix a cube $Q_0\subset\re^n$ and let $0\le F\in L^1(Q_0)$.
Assume that there are constants $\Theta\ge 1$ and  $0\le\delta<2^{-1}$ such that for every $Q \in \dd(Q_0)\setminus\{Q_0\}$ there exist non-negative functions $H^Q$, $G^Q$ and
a constant $g^Q\ge 0$ satisfying
\begin{enumerate}\itemsep=0.3cm
\item[(i)] $F(x) \leq G^Q(x)+H^Q(x)$ for a.e.~$x \in Q$,

\item[(ii)]  $\displaystyle\| H^Q \|_{L^\infty(Q)} \leq \Theta \mvint_{\widehat{Q}} F(x)\,dx$, 

\item[(iii)] $\displaystyle\mvint_Q G^Q(x)\, dx \leq \delta \mvint_{\widehat{Q}} F(x)\,dx + g^Q$.
\end{enumerate}
Define
$$
G_{Q_0}^*(x):
=\sup_{x\in Q\in\dd(Q_0)} g^Q.
$$
Given $\lambda\ge\aver_{Q_0} F(x)\,dx$, for  every $K>\Theta$ and $0<\gamma<1$ we have
\begin{equation}\label{eqn:good-lambda}
\big|\{x\in Q_0: \M_{Q_0} F(x)> K\lambda, G_{Q_0}^*(x)\le\lambda\gamma \}\big|
\le
\frac{\delta+\gamma}{K-\Theta}\big|\{x\in Q_0: \M_{Q_0} F(x)> \lambda\}\big|.
\end{equation}
Let $1< p<1+\frac{\log(1/(2\,\delta))}{\,\log(2\,\Theta)}$ \textup{(}notice that if $\delta=0$ we can take any $p>1$\textup{)}, then
\begin{equation}\label{eqn:good-lambda:weak-Lp}
\|F\|_{L^{p,\infty}, Q_0}
\le
\|\M_{Q_0} F\|_{L^{p,\infty}, Q_0}
\le C_{p,\Theta,\delta}\|G_{Q_0}^*\|_{L^{p,\infty}, Q_0}+ C_{p,\Theta,\delta}\mvint_{Q_0} F(x)\,dx
\end{equation}
and
\begin{equation}\label{eqn:good-lambda:Lp}
\|F\|_{L^{p}, Q_0}
\le
\|\M_{Q_0} F\|_{L^{p}, Q_0}
\le C_{p,\Theta,\delta}\|G_{Q_0}^*\|_{L^{p}, Q_0}+ C_{p,\Theta,\delta}\mvint_{Q_0} F(x)\,dx.
\end{equation}
\end{theorem}

Note that \eqref{eqn:good-lambda:weak-Lp} and \eqref{eqn:good-lambda:Lp} are non-trivial only if $p>1$, that is why we only consider this range.

Assuming this result we are going to derive applications to the John-Niren\-berg, Franchi-P\'erez-Wheeden and Gurov-Reshetnyak conditions.
We have split these in two sections: one where we use ``classical'' oscillations (see Section \ref{section:App-I}) and another where we use some ``generalized oscillations'' (see Section \ref{section:App-II}).

\section{Applications I: Classical oscillations}\label{section:App-I}

As an application of Theorem \ref{theor:good-lambda:dyadic} we shall give new transparent and simple proofs of  three
known results, see Corollaries \ref{corol:JNp}, \ref{corol:Dp-dyadic} and \ref{corol:GR} below.

\subsection{John-Nirenberg spaces}

We first recall the definition of the John-Nirenberg space. Let $1<p<\infty$. For a cube $Q\subset \re^n$ and $f\in L^1(Q)$, we denote
\begin{multline}\label{JNp}
\|f\|_{JN_p^\dya,Q}
:=
\sup \left(\frac1{|Q|}\sum_i \left(\aver_{Q_i} |f(x)-f_{Q_i}|\,dx\right)^{p}|Q_i|\right)^{1/p}
\\ 
=
\sup \left\|
\sum_i \left(\aver_{Q_i} |f(x)-f_{Q_i}|\,dx\right)\chi_{Q_i}
\right\|_{L^p,Q},
\end{multline}
where the suprema are taken over all pairwise disjoint subfamilies $\{Q_i\}_i$ of $\dd(Q)$.

Given a cube $Q_0\subset \re^n$, we say that $f\in JN_p^\dya(Q_0)$, if $f\in L^1(Q_0)$ and
\begin{align}\label{JNp-qq}
\|f\|_{JN_p^\dya(Q_0)}
:&=\sup_{Q\in\dd(Q_0)} \|f\|_{JN_p^\dya,Q}
<\infty.
\end{align}

The next result gives the embedding $JN_p^\dya(Q) \hookrightarrow L^{p,\infty}(Q)$.

\begin{corollary}\label{corol:JNp}
Given $1<p<\infty$, there exists a constant $C$ (depending only on $p$ and $n$) such that for every cube $Q\subset\re^n$ and $f\in L^1(Q)$, we have
\begin{equation}\label{JNp:weak:Lp}
\|f-f_{Q}\|_{L^{p,\infty},Q}
\le
C
\|f\|_{JN_p^\dya, Q}.
\end{equation}
\end{corollary}

\begin{proof}
Fix $Q_0\subset \re^n$ and assume that $f\in L^1(Q_0)$ satisfies $\|f\|_{JN_p^\dya, Q_0}<\infty$.
We shall apply Theorem \ref{theor:good-lambda:dyadic} to the
function $F(x):=|f(x)-f_{Q_0}|$. Since $f$ belongs to $L^1(Q_0)$, so does
$F$. Take $Q\in\dd(Q_0)\setminus\{Q_0\}$ and write
$$
F(x)
=
|f(x)-f_{Q_0}|
\le
|f(x)-f_Q|+|f_Q-f_{Q_0}|
=: G^Q(x)+H^Q(x).
$$
Note that
$$
\|H_Q\|_{L^\infty(Q)}
=
|f_Q-f_{Q_0}|
\le
\aver_{Q}|f(x)-f_{Q_0}|\,dx
\le
2^n \aver_{\widehat{Q}}F(x)\,dx,
$$
which is assumption (ii) in Theorem \ref{theor:good-lambda:dyadic} with $\Theta=2^n$. Besides,
$$
\aver_{Q} G^Q(x)\,dx
=
\aver_{Q} |f(x)-f_Q|\,dx
=:g^Q,
$$
which is  assumption (iii) in Theorem \ref{theor:good-lambda:dyadic} with $\delta=0$. Note that
$$
G_{Q_0}^*(x)
=
\sup_{x\in Q\in\dd(Q_0)} g^Q
=\sup_{x\in Q\in\dd(Q_0)} \aver_{Q} |f(x)-f_Q|\,dx
=
\M^\#_{Q_0} f(x),
$$
which is the dyadic and localized sharp maximal function.
We can then apply Theorem \ref{theor:good-lambda:dyadic} (with any $p>1$ since $\delta=0$) and obtain
\begin{multline*}
\|f-f_{Q_0}\|_{L^{p,\infty},Q_0}
\lesssim
\|G_{Q_0}^*\|_{L^{p,\infty},Q_0}+ F_{Q_0}
\\
=
\|\M^\#_{Q_0} f\|_{L^{p,\infty},Q_0}+\aver_{Q_0}|f(x)-f_{Q_0}|\,dx
\le
2\|\M^\#_{Q_0} f\|_{L^{p,\infty},Q_0}.
\end{multline*}
This is a well-known inequality, and
Theorem \ref{theor:good-lambda:dyadic} is partly motivated by it. Once we have that, we obtain the desired embedding by a standard stopping-time argument. Consider the distribution set $\Omega_\lambda = \{x\in Q_0: \M^\#_{Q_0}f(x) >\lambda\}$ with $\lambda>0$.
First consider the case
\[
\lambda\ge \lambda_0:=\aver_{Q_0}|f(x)-f_{Q_0}|\,dx.
\]
Subdivide dyadically $Q_0$ and stop
whenever
\[
\aver_{Q}|f(x)-f_Q|\,dx>\lambda.
\]
This defines a family of Calder\'on-Zygmund cubes $\{Q_i\}_i\subset\dd(Q_0)\setminus\{Q_0\}$
which are maximal, and therefore pairwise disjoint,  with respect to
the stopping criterion. By our choice of $\lambda$ they are proper subcubes of
 $Q_0$. Notice that $\Omega_\lambda=\cup_i Q_i$. Then, using that $Q_i$ is one of the stopping cubes we have
$$
\lambda^p\frac{|\Omega_\lambda|}{|Q_0|}
=
\frac{\lambda^p}{|Q_0|}\sum_i |Q_i|
=
\frac1{|Q_0|}\sum_i \left(\aver_{Q_i} |f(x)-f_{Q_i}|\,dx\right)^{p}|Q_i|
\le
\|f\|_{JN_p^\dya, Q_0}^p,
$$
since $\{Q_i\}_i\subset\dd(Q_0)$ is a pairwise disjoint family.

Consider now the case $0<\lambda\le \lambda_0$ and note  that by definition of the $JN^\dya_p$ norm we immediately have
$$
\lambda^p\frac{|\Omega_\lambda|}{|Q_0|}
\le
\lambda^p
\le
\left(\aver_{Q_0}|f(x)-f_{Q_0}|\,dx\right)^p
\le
\|f\|_{JN_p^\dya, Q_0}^p.
$$
Gathering the two cases and taking the supremum in $\lambda>0$, we conclude that
$$
\|\M^\#_{Q_0} f\|_{L^{p,\infty},Q_0}
\le
\|f\|_{JN_p^\dya, Q_0}
$$
and thus
$$
\|f-f_{Q_0}\|_{L^{p,\infty},Q_0}
\lesssim
\|\M^\#_{Q_0} f\|_{L^{p,\infty},Q_0}
\le
\|f\|_{JN_p^\dya, Q_0}.
$$
\end{proof}

In the previous proof we have obtained
$$
\|f-f_{Q_0}\|_{L^{p,\infty},Q_0}
\lesssim
\|\M^\#_{Q_0} f\|_{L^{p,\infty},Q_0}
\le
\|f\|_{JN_p^\dya, Q_0}
$$
and one may ask whether we can reverse any of the previous inequalities.
Since $1<p<\infty$ we have that $\M_{Q_0}$ is bounded on $L^{p,\infty}(Q_0)$ and therefore
$$
\|\M^\#_{Q_0} f\|_{L^{p,\infty},Q_0}
\le
2\|\M_{Q_0} (f-f_{Q_0})\|_{L^{p,\infty},Q_0}
\lesssim
\|f-f_{Q_0}\|_{L^{p,\infty},Q_0}.
$$
On the other hand, in general
the inclusion $JN_p^\dya(Q_0) \hookrightarrow L^{p,\infty}(Q_0)$ is strict.
In $\re$ we take $Q_0=[0,1)$ and $f(x)=x^{-1/p}\chi_{Q_0}(x)$.
It is straightforward to see that $f \in L^{p,\infty}(Q_0)$ but
$f\notin JN_p^\dya(Q_0)$. For the details we refer to \cite{ABKY}.

\subsection{Franchi-P\'erez-Wheeden self-improvement}
Let us consider a functional $a:\dd(Q_0)\longrightarrow [0,\infty)$. For $1<p<\infty$ and $Q\in\dd(Q_0)$ we set
$$
\|a\|_{D_p,Q}
=
\frac1{a(Q)}\sup
\left(\frac1{|Q|}\sum_{i} a(Q_i)^p|Q_i|\right)^{1/p}
=
\sup
\frac{\displaystyle\Big\|\sum_i a(Q_i)\,\chi_{Q_i}\Big\|_{L^p,Q}}
{\displaystyle\Big\|a(Q)\,\chi_{Q}\Big\|_{L^p,Q}},
$$
where the suprema are taken over all pairwise disjoint families $\{Q_i\}_i\subset \dd(Q)$.
We say that $a\in D_p^\dya(Q_0)$ provided
$$
\|a\|_{D_p^\dya(Q_0)}:=\sup_{Q\in\dd(Q_0)} \|a\|_{D_p,Q}<\infty.
$$

We are going to show that the following self-improvement result in \cite{FPW} is a
straightforward consequence of Corollary \ref{corol:JNp}.

\begin{corollary}\label{corol:Dp-dyadic}
Fix a cube $Q_0\subset\re^n$. Let $f\in L^1(Q_0)$  be such that
\begin{equation}\label{Poincare:gener:corol}
\aver_Q|f(x)-f_Q|\,dx
\le
a(Q),
\end{equation}
for every $Q\in\dd(Q_0)$. Here $a$ is a functional \textup{(}depending possibly on $f$\textup{)} as above.
Let $1<p<\infty$.
If $a\in D_p^\dya(Q_0)$, then for every $Q\in\dd(Q_0)$, we have
\begin{equation}\label{Poincare:gener:corol:Lp}
\|f-f_Q\|_{L^{p,\infty},Q}
\lesssim
\|a\|_{D_p^\dya(Q_0)}\,
a(Q).
\end{equation}
\end{corollary}

\begin{proof}
Fix $Q\in\dd(Q_0)$ and observe that \eqref{Poincare:gener:corol} implies
\begin{align*}
\|f\|_{JN_p^\dya,Q}
&=
\sup_{\{Q_i\}_i\subset\dd(Q)} \left(\frac1{|Q|}\sum_i \left(\aver_{Q_i} |f(x)-f_{Q_i}|\,dx\right)^{p}|Q_i|\right)^{1/p}
\\
&\le
\sup_{\{Q_i\}_i\subset\dd(Q)} \left(\frac1{|Q|}\sum_i a(Q_i)^p|Q_i|\right)^{1/p}
\\
&\le
\|a\|_{D_p,Q}\,a(Q)\le
\|a\|_{D_p^\dya(Q_0)}a(Q).
\end{align*}
This and \eqref{JNp:weak:Lp} immediately give
$$
\|f-f_{Q}\|_{L^{p,\infty},Q}
\lesssim
\|f\|_{JN_p^\dya, Q}
\lesssim
\|a\|_{D_p^\dya(Q_0)}a(Q).
$$
\end{proof}

\begin{remark}
In \cite{MaPe} exponential self-improvement results are obtained as follows. Assuming that $f$ satisfies \eqref{Poincare:gener:corol} with $a$ quasi-increasing (i.e., $a(Q_1)\le C_aa(Q_2)$ for $Q_1\subset Q_2$ and $Q_1, Q_2\in\dd(Q_0)$), then
$$
\|f-f_{Q}\|_{{\rm exp} L,Q}\le Ca(Q),
$$
for every $Q\in\dd(Q_0)$.
As in the  previous proof we can easily obtain such an estimate from the classical John-Nirenberg inequality:
\begin{multline*}
\|f-f_{Q}\|_{{\rm exp} L,Q}
\lesssim
\|f\|_{\bmo^{\rm dyadic }(Q)}
=
\sup_{Q'\in\dd(Q)} \aver_{Q'} |f(x)-f_{Q'}|\,dx
\\
\lesssim
\sup_{Q'\in\dd(Q)} a(Q')
\le
C_a\,
a(Q).
\end{multline*}

\end{remark}

\subsection{Gurov-Reshetnyak classes} Our last application in classical oscillations
is a new proof of the  self-improvement of the dyadic Gurov-Reshetnyak condition with the expected asymptotical behavior as $\varepsilon \to 0^+$.
Recall that given $0\le w\in L^1(Q_0)$ we say that
$w\in GR_\varepsilon^\dya(Q_0)$, where $0<\varepsilon<2$, if
\begin{equation}\label{GR}
\aver_Q|w(x)-w_Q|\,dx
\le
\varepsilon\, w_Q,
\qquad\qquad
\end{equation}
for every $Q\in\dd(Q_0)$.

\begin{corollary}\label{corol:GR}
Fix a cube $Q_0\subset\re^n$. If $w\in GR_\varepsilon^\dya(Q_0)$ with
$\varepsilon>0$ small enough \textup{(}for instance, $0<\varepsilon < 2^{-(n+1)}$\textup{)},
there exists $p(\varepsilon)>1$  such that
for every $1\le p<p(\varepsilon)$,
\begin{equation}\label{GR-p-dyadic}
\left(\aver_Q|w(x)-w_Q|^p\,dx\right)^{1/p}
\le
C\,\varepsilon\, w_Q
\end{equation}
for every $Q\in\dd(Q_0)$, and where
$C$ depends only  $n$ and $p$. Moreover,
 $p(\varepsilon)\to\infty$ as $\varepsilon\to 0^+$.
Therefore, $w\in RH_p^{\dya}(Q_0)$
for every $1\le p<p(\varepsilon)$, that is, $w$ satisfies the reverse H\"older inequality
\begin{equation}\label{RH-p-dyadic}
\left(\aver_Q w(x)^p \,dx\right)^{1/p}
\lesssim
\aver_{Q} w(x)\,dx
\end{equation}
for every $Q\in\dd(Q_0)$.
\end{corollary}

\begin{proof}
Clearly, it is enough to obtain \eqref{GR-p-dyadic} for $Q_0$ itself, since $GR_\varepsilon^\dya(Q_0)\subset GR_\varepsilon^\dya(Q)$ for every $Q\in \dd(Q_0)$.
We wish to apply Theorem \ref{theor:good-lambda:dyadic} to the function
 $F(x):=|w(x)-w_{Q_0}|$. For any $Q\in\dd(Q_0)\setminus \{Q_0\}$ we have
$$
F(x)
=
|w(x)-w_{Q_0}|
\le
|w(x)-w_Q|+|w_Q-w_{Q_0}|
=: G^Q(x)+H^Q(x).
$$
Note that
$$
\|H_Q\|_{L^\infty(Q)}
=
|w_Q-w_{Q_0}|
\le
\aver_{Q}|w(x)-w_{Q_0}|\,dx
\le
2^n \aver_{\widehat{Q}}F(x)\,dx,
$$
which gives (ii) in Theorem \ref{theor:good-lambda:dyadic} with $\Theta=2^n$. By \eqref{GR} we obtain
$$
\aver_{Q} G^Q(x)\,dx
=
\aver_{Q} |w(x)-w_Q|\,dx
\le
\varepsilon w_Q
\le
2^n\varepsilon\aver_{\widehat{Q}}F(x)\,dx
+
\varepsilon\, w_{Q_0},
$$
which is (iii) in Theorem \ref{theor:good-lambda:dyadic} with
$\delta=2^n\varepsilon$
and $g^Q=\varepsilon\, w_{Q_0}$ and hence $G_{Q_0}^*\equiv \varepsilon w_{Q_0}$, a constant function.
Assuming that $0<\varepsilon<2^{-(n+1)}$ (i.e., $0<\delta<2^{-1}$), set
$$
p(\varepsilon)
=
1+
\frac{\log(1/(2\,\delta))}{\,\log(2\,\Theta)}
=
\frac1{n+1}\frac{\log (\varepsilon^{-1})}{\log 2}>1.
$$
Observe that $p(\varepsilon)\to\infty$ as $\varepsilon\to 0^+$.
If we now take $1\le p<p(\varepsilon)$, \eqref{eqn:good-lambda:Lp} gives as desired \eqref{GR-p-dyadic}:
\begin{multline*}
 \|w-w_{{Q}_0}\|_{L^p,Q_0}
\leq
\|\M_{Q_0}F\|_{L^{p},Q_0} \lesssim \|G_{Q_0}^*\|_{L^{p},Q_0} + F_{Q_0}
\\
=
\varepsilon\, w_{Q_0}
+\mvint_{Q_0} |w-w_{{Q}_0}|\,dx \lesssim
\varepsilon\, w_{Q_0}.
\end{multline*}
To complete the proof we just observe that
\eqref{GR-p-dyadic} and the triangle inequality immediately imply \eqref{RH-p-dyadic}.
\end{proof}

\begin{remark}
An analogous argument gives a similar
self-improvement for the \emph{weak} dyadic
Gurov-Reshetnyak condition
$$
\aver_Q|w(x)-w_Q|\,dx
\le
\varepsilon w_{\widehat{Q}}
$$
for every $Q\in\dd(Q_0) \setminus \{Q_0\}$,
if $\varepsilon$ is small enough. Recall that  $\widehat{Q}$ is the dyadic parent of $Q$.
This weak condition will be studied in a more general setting
in Part II.
\end{remark}

\section{Applications II: Generalized oscillations}\label{section:App-II}

Our next goal is to show that our good-$\lambda$ inequality allows us to consider other oscillations as well. We show that the previous applications can be translated into  a new context, where the classical oscillation $f-f_Q$ is replaced by another oscillation $B_Q f$ satisfying some conditions. In the present situation, and in view of the local character of the good-$\lambda$ inequality, all operators will be local. This should be compared with \cite{BeM}, where non-local oscillations are considered.

\begin{definition}\label{defi:local-osci}
Given a cube $Q_0\subset\re^n$ we say that the family $\mathbb{B}_{Q_0}:=\{B_Q\}_{Q\in\dd(Q_0)}$ is a local oscillation if, after setting $A_Q:=I-B_Q$, the following conditions hold:
\begin{enumerate}\itemsep=0.2cm
\item[(a)] For every $Q\in\dd(Q_0)$, $A_Q$ is a linear operator acting on functions in $L^1(Q_0)$.

\item[(b)] For every $Q\in\dd(Q_0)$ we have
\[
\|A_Q f\|_{L^\infty(Q)}\le C_{\mathbb{B}} \aver_{Q}|f(y)|\,dy.
\]

\item[(c)] For every $Q_1$, $Q_2\in\dd(Q_0)$ satisfying $Q_1\subset Q_2$ we have $B_{Q_1} A_{Q_2} f=0$ a.e. in $Q_1$ (equivalently, $A_{Q_1}A_{Q_2} f= A_{Q_2} f$ a.e. in $Q_1$).
\end{enumerate}

Notice that (a) and (b) imply that $(A_Q f)\chi_Q=A_Q(f\chi_Q)\chi_Q$ and that is why we say that the family is local.

\begin{example}\label{example1}
Set $A_Qf=f_Q\chi_Q$ and $B_Q=I-A_Q$. Then $\mathbb{B}_{Q_0}:=\{B_Q\}_{Q\in\dd(Q_0)}$ is clearly a local oscillation
\end{example}

\begin{example}
As in \cite{FPW}, for a fixed $m\ge 0$, we let $\mathcal{P}_m$ be the space of real-valued polynomials of degree at most $m$ which is generated by the linearly independent collection of polynomials $\mathcal{S}_m=\{x^\alpha\}_{|\alpha|\le m}$ where $\alpha=(\alpha_1,\dots,\alpha_n)$ is a multi-index and $x^\alpha=x_1^{\alpha_1}\cdots x_n^{\alpha_n}$. Let $Q_0$ be the cube with the center at the origin and side length $1$ and endow
$\mathcal{P}_m$ with the inner product
\[
\langle f, g\rangle_{Q_0}=\aver_{Q_0}fg\,dx=\int_{Q_0}fg\,dx.
\]
Then $(\mathcal{P}_m, \langle\cdot,\cdot\rangle_{Q_0})$ is a finite-dimensional Hilbert space. Using the Gram-Schmidt methods on $\mathcal{S}_m$ we can find $\mathcal{B}_m=\{\varphi_\alpha\}_{|\alpha|\le m}$, an orthonormal basis of $(\mathcal{P}_m, \langle\cdot,\cdot\rangle_{Q_0})$. Notice that since the space $\mathcal{P}_m$ is finite-dimensional, all norms on it are equivalent and therefore for every $\alpha$ with $|\alpha|\le m$, we have
$$
\|\varphi_\alpha\|_{L^\infty(Q_0)}
\le
C_m\|\varphi_\alpha\|_{L^2,Q_0}
=
C_m\langle \varphi_\alpha, \varphi_\alpha\rangle_{Q_0}
=
C_m.
$$
Let $Q$ be an arbitrary cube centered at $x_Q$ and whose side length is $\ell(Q)$. We set $\varphi_\alpha^Q(x)=\varphi((x-x_Q)/\ell(Q))$ for every $x\in Q$. It is straightforward to show that $\mathcal{B}_m^{Q}=\{\varphi_\alpha^Q\}_{|\alpha|\le m}$ is an orthonormal basis of the finite-dimensional  $(\mathcal{P}_m, \langle\cdot,\cdot\rangle_{Q})$ where the inner product is now
\[
\langle f, g\rangle_Q=\aver_{Q}fg\,dx.
\]
Also, it is trivial to check that for every $\alpha$ with $|\alpha|\le m$, we have
\begin{equation}\label{varphiQ-Linfty}
\|\varphi_\alpha^Q\|_{L^\infty(Q)}
=
\|\varphi_\alpha\|_{L^\infty(Q_0)}
\le
C_m
.
\end{equation}
We now set
$$
A_Q f(x)
=
\sum_{|\alpha|\le m} \langle f,\varphi_\alpha^Q\rangle_Q\varphi_\alpha^Q(x)\chi_Q(x).
$$
We shall see that if $B_Q=I-A_Q$, then $\mathbb{B}_{Q_0}=\{B_Q\}_{Q\in\dd(Q_0)}$ is a local oscillation. Notice that \eqref{varphiQ-Linfty} implies that $A_Q$ is a linear operator, well-defined for every $f\in L^1(Q_0)$ and it satisfies
\begin{multline*}
\|A_Q f\|_{L^\infty(Q)}
\le
\sum_{|\alpha|\le m} \big|\langle f,\varphi_\alpha^Q\rangle_Q\big|\|\varphi_\alpha^Q\|_{L^\infty(Q)}
\\
\le
\sum_{|\alpha|\le m} \|\varphi_\alpha^Q\|_{L^\infty(Q)}^2\aver_Q|f(y)|\,dy
\le
C_m' \aver_Q|f(y)|\,dy.
\end{multline*}
We finally check the item (c). We notice that $A_Q$ is a projection from $L^2(Q)$ onto the collection of polynomials of $\mathcal{P}_m$ restricted to $Q$. In particular, if $\pi\in \mathcal{P}_m$ then $A_Q\pi=\pi\chi_Q$. Thus, if $Q_1, Q_2\in\dd(Q_0)$ with $Q_1\subset Q_2$, we have that
$A_{Q_2} f=\pi_f\chi_{Q_2}$ with $\pi_f\in \mathcal{P}_m$, and, for every $x\in Q_1$,
$$
A_{Q_1} A_{Q_2} f(x)
=
A_{Q_1} (\pi_f\chi_{Q_2})(x)
=
A_{Q_1} (\pi_f)(x)
=
\pi_f(x)\chi_{Q_1}(x)
=
A_{Q_2} f(x),
$$
which is the desired property.

We notice that if $m=0$, then $\varphi_\alpha^Q\equiv 1$ and $A_Q f=f_Q\chi_Q$, and we are back at Example \ref{example1}.
\end{example}
\end{definition}

\subsection{John-Nirenberg spaces  for local oscillations}

Let $Q_0\subset\re^n$ be a cube  and $\mathbb{B}_{Q_0}$be  a local oscillation.
We say that $f\in JN_{\mathbb{B}_{Q_0}, p}^\dya(Q_0)$, provided $f\in L^1(Q_0)$ and
\begin{align}\label{JNp:dya:B-aa}
\|f\|_{JN_{\mathbb{B}_{Q_0}, p}^\dya(Q_0)}
:&=\sup_{Q\in\dd(Q_0)} \|f\|_{JN_{\mathbb{B}_{Q_0}, p}^\dya, Q}<\infty,
\end{align}
where
\begin{multline}\label{JNp:dya:B}
\|f\|_{JN_{\mathbb{B}_{Q_0}, p}^\dya,Q}
:=
\sup\left(\frac1{|Q|}\sum_i \left(\aver_{Q_i} |B_{Q_i}f(x)|\,dx\right)^{p}|Q_i|\right)^{1/p}
\\ 
=
\sup \left\|
\sum_i \left(\aver_{Q_i} |B_{Q_i}f(x)|\,dx\right)\chi_{Q_i}
\right\|_{L^p,Q}.
\end{multline}
Here the suprema are taken over all pairwise disjoint families $\{Q_i\}_i\subset\dd(Q)$.
We show that $JN_{\mathbb{B}_{Q_0}, p}^\dya(Q_0) \hookrightarrow L^{p,\infty}(Q_0)$ whenever $1<p<\infty$.

\begin{corollary}\label{corol:JNp:B}
Let  $1<p<\infty$. Fix a cube $Q_0\subset\re^n$ and let $\mathbb{B}_{Q_0}$ be a local oscillation.
For every  $Q\in\dd(Q_0)$ and $f\in L^1(Q)$ we have
\begin{equation}\label{JNp:weak:Lp:B}
\|B_{Q} f\|_{L^{p,\infty},Q}
\le
C_{\mathbb{B}}
\|f\|_{JN_{\mathbb{B}_{Q_0}, p}^\dya, Q}.
\end{equation}
\end{corollary}

\begin{proof}
We shall apply Theorem \ref{theor:good-lambda:dyadic}. Fix $Q_0\subset \re^n$, $f\in L^1(Q_0)$ and assume that $\|f\|_{JN_{\mathbb{B}_{Q_0}, p}^\dya, Q}<\infty$. Note that directly from the definition
$$
\aver_{Q_0}|B_{Q_0}f(x)|\,dx
\le
\|f\|_{JN_{\mathbb{B}_{Q_0}, p}^\dya, Q_0}<\infty.
$$
Thus $F(x):=|B_{Q_0}f(x)|\in L^1(Q_0)$. For every  $Q\in\dd(Q_0)\setminus\{Q_0\}$ we have
$$
F(x)
=
|B_{Q_0}f(x)|
\le
|B_Q f(x)|+ |A_Q f (x)-A_{Q_0}f(x)|
=: G^Q(x)+H^Q(x).
$$
By Definition \ref{defi:local-osci} it follows that
\begin{multline*}
\|H_Q\|_{L^\infty(Q)}
=
\|A_Q f -A_{Q_0}f\|_{L^\infty(Q)}
=
\|A_Qf  -A_Q A_{Q_0}f\|_{L^\infty(Q)}
\\
=
\|A_Q B_{Q_0}f \|_{L^\infty(Q)}
\le
C_{\mathbb{B}}\aver_{Q}|B_{Q_0}w(x)|\,dx
\le
2^nC_{\mathbb{B}} \aver_{\widehat{Q}}F(x)\,dx,
\end{multline*}
which is (ii) in Theorem \ref{theor:good-lambda:dyadic} with $\Theta=2^nC_{\mathbb{B}}$.
Moreover,
$$
\aver_{Q} G^Q(x)\,dx
=
\aver_{Q} |B_Qf(x)|\,dx
=:g^Q
$$
which is (iii)  in Theorem \ref{theor:good-lambda:dyadic} with $\delta=0$. Note that
$$
G_{Q_0}^*(x)
=
\sup_{x\in Q\in\dd(Q_0)} g^Q
=\sup_{x\in Q\in\dd(Q_0)} \aver_{Q} |B_Qf(x)|\,dx
=
:\M^\#_{\mathbb{B}_{Q_0}} f(x),
$$
which is the dyadic and localized sharp maximal function associated with the oscillation $\mathbb{B}_{Q_0}$.
Theorem \ref{theor:good-lambda:dyadic} (with any $1<p<\infty$, since $\delta=0$) implies
\begin{multline*}
\|B_{Q_0} f\|_{L^{p,\infty},Q_0}
\lesssim
\|G_{Q_0}^*\|_{L^{p,\infty},Q_0}+ F_{Q_0}
\\
=
\|\M^\#_{\mathbb{B}_{Q_0}} f\|_{L^{p,\infty},Q_0}+\aver_{Q_0}|B_{Q_0}f(x)|\,dx
\le
2\|\M^\#_{\mathbb{B}_{Q_0}} f\|_{L^{p,\infty},Q_0}.
\end{multline*}

To complete the proof, let
$$
\lambda\ge \lambda_0:=\aver_{Q_0}|B_{Q_0}f(x)|\,dx,
$$
and consider the distribution set
$
\Omega_\lambda
=
\{x\in Q_0: \M^\#_{\mathbb{B}_{Q_0}}f(x)>\lambda\}
$.
Subdivide dyadically $Q_0$ and stop whenever 
\begin{equation}
\aver_{Q}|B_Q f|\,dx>\lambda.
\label{eq:stop-crite}
\end{equation}
This defines a family of Calder\'on-Zygmund cubes $\{Q_i\}_i\subset\dd(Q_0)$ which are maximal, and therefore pairwise disjoint,  with respect to the stopping criterion. By our choice of $\lambda$ we have that $\{Q_i\}_i\subset\dd(Q_0)\setminus\{Q_0\}$. Notice that $\Omega_\lambda=\cup_i Q_i$. Then, using that $Q_i$ satisfies \eqref{eq:stop-crite} we have
$$
\lambda^p\frac{|\Omega_\lambda|}{|Q_0|}
=
\frac{\lambda^p}{|Q_0|}\sum_i |Q_i|
\le
\frac1{|Q_0|}\sum_i \left(\aver_{Q_i} |B_{Q_i}f(x)|\,dx\right)^{p}|Q_i|
\le
\|f\|_{JN^\dya_{\mathbb{B}_{Q_0},p}, Q_0}^p,
$$
since $\{Q_i\}_i\subset\dd(Q_0)$ is a pairwise disjoint family. On the other hand, if $0<\lambda<\lambda_0$, by the definition of the $JN^\dya_{\mathbb{B}_{Q_0},p}$ norm we immediately have
$$
\lambda^p\frac{|\Omega_\lambda|}{|Q_0|}
\le
\lambda^p
\le
\left(\aver_{Q_0}|B_{Q_0}f(x)|\,dx\right)^p
\le
\|f\|_{JN^\dya_{\mathbb{B}_{Q_0},p}, Q_0}^p.
$$
Collecting the two cases and taking the supremum in $\lambda>0$ we conclude that
$$
\|\M^\#_{\mathbb{B}_{Q_0}}f\|_{L^{p,\infty},Q_0}
\le
\|f\|_{JN^\dya_{\mathbb{B}_{Q_0},p}, Q_0}
$$
and thus
$$
\|B_{Q_0}f\|_{L^{p,\infty},Q_0}
\lesssim
\|\M^\#_{\mathbb{B}_{Q_0}}f\|_{L^{p,\infty},Q_0}
\le
\|f\|_{JN^\dya_{\mathbb{B}_{Q_0},p}, Q_0}.
$$
\end{proof}

Note that the previous corollary implies that $JN_{\mathbb{B}_{Q_0}, p}^\dya(Q_0) \hookrightarrow L^{p,\infty}(Q_0)$, because
\begin{align*}
\|f\|_{L^{p,\infty},Q_0}
\le
\|B_{Q_0}f\|_{L^{p,\infty},Q_0}+\|A_{Q_0}f\|_{L^{p,\infty},Q_0}
\lesssim
\|f\|_{JN^\dya_{\mathbb{B}_{Q_0},p}, Q_0}+\aver_{Q_0}|f(x)|\,dx<\infty.
\end{align*}

\subsection{Franchi-P\'erez-Wheeden self-improvement for local oscillations}

As before we immediately obtain the following consequence of Corollary \ref{corol:JNp:B}.

\begin{corollary}\label{corol:Dp-dyadic:B}
Fix a cube $Q_0\subset\re^n$ and a local oscillation $\mathbb{B}_{Q_0}$. Let $f\in L^1(Q_0)$  be such that
\begin{equation}\label{Poincare:gener:corol:B}
\aver_Q|B_Q f(x)|\,dx
\le
a(Q)
\end{equation}
for every $Q\in\dd(Q_0)$,
where $a:\dd(Q_0)\longrightarrow [0,\infty)$ is a functional \textup{(}depending probably on $f$\textup{)}.
Let $1<p<\infty$. If  $a\in D_p^\dya(Q_0)$, then, for every $Q\in\dd(Q_0)$,
\begin{equation}\label{Poincare:gener:corol:Lp:B}
\|B_Q\|_{L^{p,\infty},Q}
\lesssim
\|a\|_{D_p^\dya(Q_0)}\,
a(Q).
\end{equation}
\end{corollary}

\begin{proof}
Fix $Q\in\dd(Q_0)$ and assume that $a(Q)<\infty$; otherwise there is nothing to prove.
We first observe that \eqref{Poincare:gener:corol:B} implies
\begin{multline*}
\|f\|_{JN^\dya_{\mathbb{B}_{Q_0},p}, Q}
=
\sup_{\{Q_i\}_i\subset\dd(Q)} \left(\frac1{|Q|}\sum_i \left(\aver_{Q_i} |B_{Q_i}f(x)|\,dx\right)^{p}|Q_i|\right)^{1/p}
\\
\le
\sup_{\{Q_i\}_i\subset\dd(Q)} \left(\frac1{|Q|}\sum_i a(Q_i)^p|Q_i|\right)^{1/p}
\le
\|a\|_{D_p^\dya(Q_0)}a(Q)<\infty.
\end{multline*}
This and \eqref{JNp:weak:Lp:B} immediately give the desired estimate
$$
\|B_{Q}f\|_{L^{p,\infty},Q}
\lesssim
\|f\|_{JN^\dya_{\mathbb{B}_{Q_0},p}, Q}
\lesssim
\|a\|_{D_p^\dya(Q_0)}a(Q).
$$
\end{proof}

\subsection{Gurov-Reshetnyak classes for local oscillations}

Fix a cube $Q_0\subset\re^n$, a local oscillation  $\mathbb{B}_{Q_0}$ and let $0\le w\in L^1(Q_0)$.
We say  that  $w\in GR^\dya_{\mathbb{B}_{Q_0},\varepsilon}(Q_0)$, if
\begin{equation}\label{GR-dyadic:B}
\aver_Q|B_Q w(x)|\,dx
=
\aver_Q|w(x)-A_Q w(x)|\,dx
\le
\varepsilon\aver_{Q}|A_Q w(x)|\,dx
\end{equation}
for every $Q\in\dd(Q_0)$.

Let us observe that one could have defined $GR^\dya_{\mathbb{B}_{Q_0},\varepsilon}(Q_0)$ with a right hand term of the form
$\varepsilon w_Q$. This would lead us to an equivalent definition (provided $\varepsilon$ is small enough). Indeed,
\eqref{GR-dyadic:B} and Definition \ref{defi:local-osci} imply that
$$
\aver_Q|B_Q w(x)|\,dx
\le
\varepsilon\aver_{Q}|A_Q w(x)|\,dx
\le
\varepsilon C_{\mathbb{B}}\,w_Q.
$$
Conversely, if
$$
\aver_Q|B_Q w|\,dx\le \varepsilon'w_Q
$$
with $0<\varepsilon'<1/2$, then
$$
\aver_Q|B_Q w(x)|\,dx
\le
\varepsilon'w_Q\,dx
\le
\frac12\aver_Q|B_Q w(x)|\,dx
+
\varepsilon'\aver_Q|A_Q w(x)|\,dx,
$$
and the first term in the last quantity can be absorbed  to obtain \eqref{GR-dyadic:B} with $\varepsilon=2\varepsilon'$.

\begin{corollary}\label{corol:GR-dyadic:B}
Fix a cube $Q_0\subset\re^n$,  a local oscillation $\mathbb{B}_{Q_0}$ and $w\in GR^\dya_{\mathbb{B}_{Q_0},\varepsilon}(Q_0)$.
If $0<\varepsilon<1$ is small enough \textup{(}for instance,  $0<\varepsilon<2^{-(n+2)}C_{\mathbb{B}}^{-1}$\textup{)}, there exists $p(\varepsilon)>1$ such that for every $1\le p<p(\varepsilon)$,
\begin{equation}\label{GR-p-dyadic:B}
\left(\aver_Q|B_Q w(x)|^p\,dx\right)^{1/p}
\le
C\varepsilon\aver_{Q}|A_Q w(x)|\,dx
\end{equation}
for every $Q\in\dd(Q_0)$.
Moreover, $p(\varepsilon)\to\infty$ as $\varepsilon\to 0^+$.
Therefore $w\in RH_p^\dya(Q_0)$ for every $1\le p<p(\varepsilon)$, that is,
\begin{equation}\label{RH-p-dyadic:B}
\left(\aver_Q w(x)^p \,dx\right)^{1/p}
\lesssim
\aver_{Q} w(x)\,dx
\end{equation}
for every $Q\in\dd(Q_0)$.
\end{corollary}

\begin{proof}
We first observe that \eqref{GR-p-dyadic:B}, the triangle inequality and Definition \ref{defi:local-osci} imply \eqref{RH-p-dyadic:B}:
\begin{multline*}
\left(\aver_Q w(x)^p \,dx\right)^{1/p}
\le
\left(\aver_Q |B_Q w(x)|^p \,dx\right)^{1/p}
+
\left(\aver_Q |A_Q w(x)|^p \,dx\right)^{1/p}
\\
\le
(1+C\,\varepsilon)\left(\aver_Q |A_Q w(x)|^p \,dx\right)^{1/p}
\le
(1+C\,\varepsilon)\,C_{\mathbb{B}} \aver_Q w(x) \,dx.
\end{multline*}

To obtain \eqref{GR-p-dyadic:B} we note that we can just prove it for $Q_0$ (for an arbitrary $Q\in\dd(Q_0)$ we simply repeat the argument with $Q$ in place of $Q_0$). We shall apply Theorem \ref{theor:good-lambda:dyadic}. By  assumption $F(x):=|B_{Q_0}w(x)|\in L^1(Q_0)$. Let $Q\in\dd(Q_0)\setminus\{Q_0\}$ and, for every $x\in Q$, we write
$$
F(x)
=
|w(x)-A_{Q_0}w (x)|
\le
|B_Q w(x)|+ |A_Qw (x)-A_{Q_0}w(x)|
=: G^Q(x)+H^Q(x)
$$
Note that by Definition \ref{defi:local-osci} we have
\begin{multline*}
\|H_Q\|_{L^\infty(Q)}
=
\|A_Qw -A_{Q_0}w\|_{L^\infty(Q)}
=
\|A_Qw -A_Q A_{Q_0}w\|_{L^\infty(Q)}
\\
=
\|A_Q B_{Q_0}w \|_{L^\infty(Q)}
\le
C_{\mathbb{B}}\aver_{Q}|B_{Q_0}w(x)|\,dx
\le
2^nC_{\mathbb{B}} \aver_{\widehat{Q}}F(x)\,dx,
\end{multline*}
which is (ii) in Theorem \ref{theor:good-lambda:dyadic} with $\Theta=2^nC_{\mathbb{B}}$. By \eqref{GR-dyadic:B} and Definition \ref{defi:local-osci} we obtain
\begin{align*}
\aver_{Q} G^Q(x)\,dx
&=
\aver_{Q} |B_Q w(x)|\,dx
\le
\varepsilon\aver_{Q}|A_Q f(x)|\,dx
\le
\varepsilon\, C_{\mathbb{B}}\aver_{Q}|f(x)|\,dx
\\
&\le
\varepsilon\, C_{\mathbb{B}}\aver_{Q}|B_{Q_0}f(x)|\,dx
+
\varepsilon\, C_{\mathbb{B}}\aver_{Q}|A_{Q_0}f(x)|\,dx
\\
&\le
2^n\varepsilon\, C_{\mathbb{B}}\aver_{\widehat{Q}}F(x)\,dx
+
\varepsilon\, C_{\mathbb{B}}^2\aver_{Q_0}|f(x)|\,dx
\\
&
\le
2^n\varepsilon\, C_{\mathbb{B}}\aver_{\widehat{Q}}F(x)\,dx
+
\varepsilon\, C_{\mathbb{B}}^2\aver_{Q_0}|B_{Q_0}f(x)|\,dx
+
\varepsilon\, C_{\mathbb{B}}^2\aver_{Q_0}|A_{Q_0}f(x)|\,dx
\\
&
\le
2^n\varepsilon\, C_{\mathbb{B}}\aver_{\widehat{Q}}F(x)\,dx
+
\varepsilon\,(1+\varepsilon)\,C_{\mathbb{B}}^2\aver_{Q_0}|A_{Q_0}f(x)|\,dx
,
\end{align*}
which is (iii)  in Theorem \ref{theor:good-lambda:dyadic} with $\delta=2^n\,\varepsilon C_{\mathbb{B}}$ and
$$
g^Q=\varepsilon(1+\varepsilon)C_{\mathbb{B}}^2\aver_{Q_0}|A_{Q_0}f(x)|\,dx.
$$
We have
$$
G_{Q_0}^*\equiv \varepsilon(1+\varepsilon)C_{\mathbb{B}}^2\aver_{Q_0}|A_{Q_0}f(x)|\,dx.
$$
Assuming that $0<\varepsilon<(2^{n+1}\,C_{\mathbb{B}})^{-1}$ (that is, $0<\delta<2^{-1}$), set
$$
p(\varepsilon)
=
1+
\frac{\log(1/(2\,\delta))}{\,\log(2\,\Theta)}
=
\frac{\log (\varepsilon^{-1})}{\log (2^{n+1}C_{\mathbb{B}})}>1,
$$
and observe that $p(\varepsilon)\to\infty$ as $\varepsilon\to 0^+$. If we now take $1\le p<p(\varepsilon)$, \eqref{eqn:good-lambda:Lp} gives
\begin{multline*}
\|B_{Q_0}w\|_{L^p,Q_0}
=
\|F\|_{L^p,Q_0}
\lesssim
\|G_{Q_0}^*\|_{L^p,Q_0}+F_{Q_0}
\\
=
\varepsilon(1+\varepsilon)C_{\mathbb{B}}^2\aver_{Q_0}|A_{Q_0}f(x)|\,dx+\aver_{Q_0} |B_{Q_0}w(x)|\,dx
\lesssim
\varepsilon\aver_{Q_0}|A_{Q_0}f(x)|\,dx.
\end{multline*}
This shows \eqref{GR-p-dyadic:B} with $Q_0$ in place of $Q$ and the proof is complete.
\end{proof}

\section{Proof of Theorem \ref{theor:good-lambda:dyadic}}\label{section:proof:good-lambda-Q}
For each $\lambda>0$ we set $\Omega_\lambda = \{x\in Q_0: \M_{Q_0} F(x)> \lambda\}$.  Fixed $\lambda\ge F_{Q_0}$, we subdivide $Q_0$ dyadically and stop whenever $F_Q>\lambda$. This defines the family of Calder\'on-Zygmund cubes $\{Q_i\}_i\subset\dd(Q_0)$ which are maximal, and therefore pairwise disjoint,  with respect to the property $F_Q>\lambda$. By our choice of $\lambda$ we have that $\{Q_i\}_i\subset\dd(Q_0)\setminus\{Q_0\}$. Notice that $\Omega_\lambda=\cup_i Q_i$.
Let $K>\Theta\ge 1$, $0<\gamma<1$ and set
$$
E_\lambda
=\{x\in Q_0: \M_{Q_0} F(x)> K\lambda, G_{Q_0}^*(x)\le\lambda\gamma \}.
$$
Note that $E_\lambda\subset \Omega_\lambda$ and thus
$$
|E_\lambda|
=
|E_\lambda\cap\Omega_\lambda|
=
\sum_i |E_\lambda\cap Q_i|.
$$
We analyze each term individually. We may assume that $E_\lambda\cap Q_i\neq\emptyset$ (otherwise there is nothing to prove). Thus there is $\bar{x}_i\in E_\lambda\cap Q_i$. Note that for every $x\in Q_i$ we have that $\M_{Q_0} F(x)=\M_{Q_i} F(x)$ since $\M_{Q_0} F(x)>\lambda$ and, by the maximality of the Calder\'on-Zygmund cubes, $F_Q\le \lambda$ for any  cube $Q\in\dd(Q_0)$ with $Q_i\subsetneq Q$. Then, for every
$x\in E_\lambda\cap Q_i$, we can use (i), (ii) and the maximality of the Calder\'on-Zygmund cubes to obtain
\begin{multline*}
K\lambda
<
\M_{Q_0} F(x)
=
\M_{Q_i} F(x)
\le
\M_{Q_i} G^{Q_i}(x)
+
\M_{Q_i} H^{Q_i}(x)
\\
\le
\M_{Q_i} G^{Q_i}(x)
+
\Theta\aver_{\widehat{Q}_i} F\,dx
\le
\M_{Q_i} G^{Q_i}(x)
+
\Theta\lambda.
\end{multline*}
The weak-type $(1,1)$ estimate for $\M_{Q_i}$, (iii) and the fact that $\bar{x}_i\in E_\lambda\cap Q_i$ imply
\begin{align*}
|E_\lambda\cap Q_i|
&\le
\big|\{
x\in Q_i: \M_{Q_i} G^{Q_i}>(K-\Theta)\lambda
\}\big|
\le
\frac1{(K-\Theta)\lambda}\int_{Q_i} G^{Q_i}(x)\,dx
\\
&\le
\frac{|Q_i|}{(K-\Theta)\lambda}\left(\delta
\aver_{\widehat{Q}_i} F(x)\,dx
+
g^{Q_i}\right)
\\
&\le
\frac{|Q_i|}{(K-\Theta)\lambda}\left(\delta\lambda
+G_{Q_0}^*(\bar{x}_i)\right)
\le
\frac{\delta+\gamma}{K-\Theta}|Q_i|.
\end{align*}
Summing on $i$ we readily obtain \eqref{eqn:good-lambda}.

We next show \eqref{eqn:good-lambda:weak-Lp}. Note first that by \eqref{eqn:good-lambda} we have
\begin{align*}
|\Omega_{K\lambda}|
\le
|E_{\lambda}|+
\big|\{x\in Q_0: G_{Q_0}^*(x)>\gamma\lambda\}\big|
\le
\frac{\delta+\gamma}{K-\Theta}|\Omega_\lambda|
+
\big|\{x\in Q_0: G_{Q_0}^*(x)>\gamma\lambda\}\big|
\end{align*}
for every $\lambda\ge F_{Q_0}$.
Thus, for every $0<\lambda<\infty$,
\begin{equation}\label{eqn:goof-lambda-levelsets-full}
|\Omega_{K\lambda}|
\le
\frac{\delta+\gamma}{K-\Theta}|\Omega_\lambda|
+
\big|\{x\in Q_0: G_{Q_0}^*(x)>\gamma\lambda\}\big|
+
|Q_0|\chi_{\{0<\lambda<F_{Q_0}\}}(\lambda)
.
\end{equation}
For every $N$, and $1< p<1+\frac{\log(1/(2\,\delta))}{\,\log(2\,\Theta)}$ the previous estimate leads to
\begin{multline*}
I_N:
=\sup_{0<\lambda<N} \lambda^p\frac{|\Omega_\lambda|}{|Q_0|}
=
K^p\sup_{0<\lambda<N/K} \lambda^p\frac{|\Omega_{K\lambda}|}{|Q_0|}
\\
\le
(2\,\Theta)^p\frac{\delta+\gamma}{\Theta} I_N
+
\frac{(2\,\Theta)^p}{\gamma^p}\|G_{Q_0}^*\|_{L^{p,\infty}, Q_0}^p
+
(2\,\Theta)^p(F_{Q_0})^p,
\end{multline*}
where we have chosen $K=2\Theta$. Let us observe  that our choice of $p$ guarantees that $(2\,\Theta)^p\delta/\Theta<1$ and hence we can take $\gamma$ small enough so that $(2\,\Theta)^p(\delta+\gamma)/\Theta<1$. This and the fact that $I_N\le N^p<\infty$ allow us to absorb the first term in the last estimate to obtain
$$
I_N
\le
C_{\Theta,\delta, p}\,\big(\|G_{Q_0}^*\|_{L^{p,\infty}, Q_0}^p
+(F_{Q_0})^p\big).
$$
Finally we let $N\to\infty$ and use the Lebesgue differentiation theorem to obtain \eqref{eqn:good-lambda:weak-Lp}.

To obtain the strong type estimates we proceed analogously. We use \eqref{eqn:goof-lambda-levelsets-full} to show that
\begin{multline*}
I_N:
=\int_0^Np\lambda^p\frac{|\Omega_\lambda|}{|Q_0|}\frac{d\lambda}{\lambda}
=
K^p\int_0^{N/K} p\lambda^p\frac{|\Omega_{K\lambda}|}{|Q_0|}\frac{d\lambda}{\lambda}
\\
\le
(2\,\Theta)^p\frac{\delta+\gamma}{\Theta} I_N
+
\frac{(2\,\Theta)^p}{\gamma^p}\|G_{Q_0}^*\|_{L^{p}, Q_0}^p
+
(2\,\Theta)^p(F_{Q_0})^p,
\end{multline*}
where again $K=2\Theta$. From here \eqref{eqn:good-lambda:Lp} follows as before. \qed

\part{Spaces of homogeneous type}\label{part:II}

\section{The good $\lambda$-inequality}\label{section:good-lambda:X}

In the sequel $X=(X,d,\mu)$ is a metric space endowed with a metric $d$ and a
Borel regular doubling measure $\mu$ with $0<\mu(B)<\infty$ for all balls $B$.
 Actually, all of our results hold true in
spaces of homogeneous type equipped with a \emph{quasi}metric (see \cite{CoWe,StTo}),
but for simplicity of presentation we concentrate
on metric spaces.
A ball means an open ball which comes with a center
and a positive finite radius, that is,
$$
B=B(x_B,r_B)=\{y\in X:d(y,x_B)<r_B\}.
$$
The $\lambda$-dilate of $B$ is
defined by $\lambda B:=B(x_B,\lambda r_B)$ and
the doubling condition means that
$
\mu (2B)\leq c\,\mu(B)
$
for all balls $B$ in $X$. This implies that there exists $D>0$ such that
\begin{equation}\label{doub1}
\frac{\mu(B')}{\mu(B)} \leq c_\mu \left( \frac{r_{B'}}{r_B} \right)^D
\end{equation}
for every $B\subset B'$ with $r_{B}\le r_{B'}$.

Given a ball $B$ and a  fixed (small) $\eta >0$, we write
$\widehat{B}:=(1+\eta)B$.
Fix  a ball $B_0=B(x_{B_0}, r_{B_0})$ and consider the following family of balls
\begin{equation}\label{def-B-cal}
\Bs:=
\Bs_{B_0}
:=
\{B=B(x_B,r_B): x_B \in B_0 \text{ with } r_B \leq \eta\, r_{B_0}\}.
\end{equation}
It should be observed that
\begin{equation}\label{dila-Fam-B}
B \in \Bs\quad\mbox{and}\quad \tau\ge 1
\qquad\text{imply}\qquad B \subset \widehat{B}_0
\quad\mbox{and}\quad \tau B \subset \tau \widehat{B}_0.
\end{equation}
In the following we consider the Hardy-Littlewood maximal operator
$M_\Bs$ with respect to the basis $\Bs$, that is,
$$
M_\Bs f(x)
:=
M_{\Bs_{B_0}} f(x)
:=
\sup_{x \in B \in \Bs} \mvint_B |f| \,d\mu,
$$
with the convention that $M_\Bs f(x)=0$ if there is no ball in $\Bs$ containing the
point $x$. In particular, $M_\Bs f$ vanishes outside of $\widehat{B}_0$.
By the Lebesgue differentiation theorem we have $|f| \leq M_\Bs f$ $\mu$-a.e. in $B_0$.

\begin{theorem}\label{theor:good-lambda:X}
Fix a ball $B_0\subset X$ and consider the family $\Bs=\Bs_{B_0}$ as before. Let $0\le F\in L^1(\widehat{B}_0)$ and assume that there are constants $\Theta\ge 1$,  $0\le\delta<2^{-1}$ and $\tau\ge 1$ such that for every  $B \in \Bs$ there exist non-negative functions $H^B$, $G^B$ and
a constant $g^B\ge 0$ satisfying
\begin{enumerate}\itemsep=0.3cm
\item[(i)] $F(x) \leq G^B(x)+H^B (x)$ for $\mu$-a.e.~$x \in B$,

\item[(ii)]  $\displaystyle\| H^B \|_{L^\infty(B)} \leq \Theta \mvint_{\tau B} F(x)\,d\mu(x)$,

\item[(iii)] $\displaystyle\mvint_B G^B(x)\, d\mu(x) \leq \delta \mvint_{\tau B} F(x)\,d\mu(x) + g^B$.
\end{enumerate}
Define
$$
G_{\Bs}^* (x):=\sup_{x \in B \in \Bs} g^B
$$
with the convention that $G_{\Bs}^* (x)=0$ if there is no ball in
$\Bs$ containing the point $x$.

There is $\lambda_0=\lambda_0(\tau, \eta, \mu)$ such that if $\lambda\ge\lambda_0\, \aver_{\widehat{B}_0} F\,d\mu$, $K>\max\{\Theta, c_\mu\,3^D\}$ and $0<\gamma<1$, then
\begin{equation}\label{eqn:good-lambda:X}
\mu\big(\{x \in \widehat{B}_0: M_{\Bs} F(x) >K\lambda,
G_{\Bs}^*(x) \leq \gamma \lambda\}\big)
\le
C_\mu
\frac{\delta+\gamma}{K-\Theta}\mu\big( \{x \in \widehat{B}_0: M_{\Bs}F(x) >\lambda\}\big),
\end{equation}
where $C_\mu\ge 1$ depends only on $\mu$.

If $0\le \delta<\frac1{2\,C_\mu}\, \min\Big\{1, \frac{\Theta}{c_\mu\,3^D}\Big\}$ and $1<p<\frac{\log(1/(C_\mu\,\delta\, \Theta^{-1}))}{\log (2\,\max\{\Theta,c_\mu\,3^D\})}$ \textup{(}notice that if $\delta=0$ we can take any $p>1$\textup{)}, then
\begin{equation}\label{main1}
\|F\|_{L^{p,\infty},B_0}
\lesssim
\|M_\Bs F\|_{L^{p,\infty},\widehat{B}_0}
\lesssim \|G_{\Bs}^* \|_{L^{p,\infty},\widehat{B}_0}+F_{\widehat{B}_0}
\end{equation}
and
\begin{equation}\label{main2}
\|F\|_{L^{p},B_0}
\le
\|M_\Bs F\|_{L^{p},\widehat{B}_0}
\lesssim \|G_{\Bs}^*\|_{L^{p},\widehat{B}_0}+F_{\widehat{B}_0}.
\end{equation}
\end{theorem}

The proof of this result is postponed until Section \ref{section:proof:good-lambda-X}. In the following section we will present some applications in the context of the John-Niren\-berg, Franchi-P\'erez-Wheeden and Gurov-Reshetnyak conditions. These are not mere translations of the ones considered in Section \ref{section:App-I} to the setting of spaces of homogeneous type as we consider oscillations in $L^\rho$ with $\rho\le 1$ in the first two applications and we allow some dilation on the right hand sides of the Gurov-Reshetnyak conditions.

\section{Applications}\label{appl}

\subsection{John-Nirenberg spaces}\label{JNpmetric}

The Euclidean definition of the John-Nirenberg spa\-ces can be generalized in a straightforward way by replacing cubes by balls. In
\cite{ABKY} a further generalization was considered by allowing the family of balls to have bounded overlap and the authors established
the corresponding embedding into the weak-$L^p$ space.

For our purposes we shall further generalize those  definitions as follows. Given  a ball $B$, $0<\rho\le 1$, $\tau \geq 1$ and $p>1$,
for each $f \in L^\rho(B)$  we write
\begin{equation}\label{jnpgen}
\|f\|_{JN_{p,\tau}^\rho, B}
:=
\sup\left(\frac1{\mu(B)} \sum_i \left( \inf_{c \in \R}
\mvint_{B_i} |f-c|^\rho \,d\mu \right)^{p/\rho}\,\mu(\tau B_i)\right)^{1/p},
\end{equation}
where the supremum runs over all pairwise disjoint families $\{\tau B_i\}_i$ with $\tau B_i\subset B$ for every $i$.
We say that $f \in JN_{p,\tau}^\rho(B_0)$ provided $f \in L^\rho(B_0)$  and
$$
\|f\|_{JN_{p,\tau}^\rho(B_0)}
:=
\sup_{B\subset B_0} \|f\|_{JN_{p,\tau}^\rho, B}<\infty.
$$
Observe that for $\tau=\rho=1$, the space coincides with the corresponding metric version of the John-Nirenberg space considered in Section \ref{section:App-I}.

We obtain  the embedding of the John-Nirenberg spaces as just defined into the corresponding weak-$L^p$ spaces.

\begin{corollary}\label{corol:JNp:X}
Given $1<p<\infty$, $0<\rho\le 1$ and  $\tau \geq 1$, there exists a constant $C$, depending only on $p$, $\tau$, $\eta$, $\rho$  and $\mu$, such that for every ball  $B\subset X$ and $f\in L^\rho(\tau \widehat{B})$, we have
\begin{equation}\label{appl1}
\|f-f_{B}\|_{L^{p,\infty}, B} \le C\, \|f\|_{JN_{p,\tau}^\rho, \tau \widehat{B}}.
\end{equation}
\end{corollary}

\begin{proof}
Fix a ball $B_0$ and assume that $\|f\|_{JN_{p,\tau}^\rho, \tau \widehat{B}_0}<\infty$ with $f\in L^\rho(\tau \widehat{B}_0)$. For every $B\subset\tau \widehat{B}_0$, let $c_B$ be the real number
for which
$$
\inf_{c \in \R} \mvint_B   |f-c|^\rho \,d\mu = \mvint_B  |f-c_B|^\rho \,d\mu.
$$
That $c_B$ exists (i.e, that the infimum is attained) follows from the fact that any sequence approximating the infimum (which is finite since $f\in L^\rho(B)$) is bounded. Then one can extract a convergent subsequence for which dominated convergence theorem can be applied.
Further details are left to the interested reader.

We shall apply Theorem \ref{theor:good-lambda:X} with $F=|f-c_{\widehat{B}_0}|^\rho\in L^1(\widehat{B}_0)$.
First, we notice that for every $B\in \Bs$
$$
F(x)
\leq |f(x)-c_{B}|^\rho+|c_{B}-c_{\widehat{B}_0}|^\rho=:G^B(x)+H^B(x).
$$
By the minimizing property of the constants $c_B$, we have
\begin{multline*}
H^B(x)
\equiv
\mvint_B |c_B-c_{\widehat{B}_0}|^\rho \,d\mu
\leq
\mvint_B |f-c_B|^\rho \,d\mu
+
\mvint_B |f-c_{\widehat{B}_0}|^\rho \,d\mu
\\
\leq
2\mvint_B |f-c_{\widehat{B}_0}|^\rho \,d\mu
=
2\mvint_B F\,d\mu
\le
2\,c_\mu\,\tau^D\,\mvint_B F\,d\mu
.
\end{multline*}
This is assumption (ii) in  Theorem \ref{theor:good-lambda:X} with $\Theta=2\,c_\mu\,\tau^D$. Besides,
$$
\mvint_B G^B\, d\mu
=
\mvint_B  |f-c_B|^\rho \,d\mu
=:
g^B,
$$
which is assumption (iii) with $\delta=0$. Note that
$G_{\Bs}^*$ is the sharp type maximal function with respect to the
basis $\Bs$ defined by
$$
G_{\Bs}^*(x)
=
\sup_{x \in B \in \Bs}
g^B
=
\sup_{x \in B \in \Bs}
\mvint_B |f-c_B|^\rho \,d\mu
=
\sup_{x \in B \in \Bs}\inf_{c \in \R} \mvint_B   |f-c|^\rho \,d\mu
.
$$
An application of \eqref{main1} with the exponent $p/\rho$ (notice that $\delta=0$ and $p/\rho\ge p>1$) gives
\begin{align*}\label{appl2}
\|f-f_{B_0}\|_{L^{p,\infty},B_0}^\rho&
\leq
2^\rho\,\|f-c_{\widehat{B}_0} \|_{L^{p,\infty},B_0}^\rho
=
2^\rho\,\|F\|_{L^{p/\rho,\infty},B_0}
\\
& \lesssim
\|G_{\Bs}^*\|_{L^{p/\rho,\infty},\widehat{B}_0}
+
\mvint_{\widehat{B}_0} |f-c_{\widehat{B}_0}|^\rho \,d\mu
\\
&
=\|(G_{\Bs}^*)^{1/\rho}\|_{L^{p,\infty}(\widehat{B}_0)}^{\rho}
+
\inf_{c\in\re} \mvint_{\widehat{B}_0} |f-c|^\rho \,d\mu
\\
& \le
\|(G_{\Bs}^*)^{1/\rho}\|_{L^{p,\infty}(\widehat{B}_0)}^{\rho}
+
\|f\|_{JN_{p,\tau}^\rho, \tau \widehat{B}_0}^\rho.
\end{align*}
To estimate $\|(G_{\Bs}^*)^{1/\rho}\|_{L^{p,\infty}(\widehat{B}_0)}$, take any
$x \in \widehat{B}_0$ with $G_{\Bs}^*(x)^{1/\rho}>\lambda$. Then there is a ball $B_x \in \Bs$, $B_x\ni x$ with
$$
\left(\mvint_{B_x} |f-c_{B_x}|^\rho \,d\mu\right)^{1/\rho} >\lambda.
$$
Now apply Vitali's covering theorem to the balls $\{\tau B_x\}_x$ to obtain
a countable family of pairwise disjoint balls $\tau B_i$ with
$$
\{x \in \widehat{B}_0: G_{\Bs}^*(x)^{1/\rho}>\lambda\} \subset \bigcup_i 5\tau B_i.
$$
Observe that since $B_i \in \Bs$, we have $\tau B_i \subset \tau \widehat{B}_0$ (see \eqref{dila-Fam-B}).
Therefore,
\begin{multline*}
\mu \{x \in \widehat{B}_0: G_{\Bs}^*(x)^{1/\rho}>\lambda \} 
\leq c_\mu 5^D \sum_i
\mu(\tau B_i)\\ \leq \frac{c_\mu 5^D}{\lambda^p} \sum_i \mu(\tau B_i)
\left( \mvint_{B_i} |f-c_{B_i}|^\rho \,d\mu \right)^{p/\rho} \leq \frac{c_\mu 5^D}{\lambda^p}
\|f\|_{JN_{p,\tau}^\rho,\tau \widehat{B}_0}^p \mu(\tau \widehat{B}_0).
\end{multline*}
Consequently,
\begin{equation*}\label{appl3}
\|(G_{\Bs}^*)^{1/\rho}\|_{L^{p,\infty}, \widehat{B}_0}
\lesssim
\|f\|_{JN_{p,\tau}^\rho,\tau \widehat{B}_0},
\end{equation*}
and the desired estimate \eqref{appl1} follows.
\end{proof}

\subsection{Franchi-P\'erez-Wheeden self-improvement}\label{FPWmetric}

Fix a ball $B_0$ and a functional $a:\{B:\ B\subset \tau\,\widehat{B}_0\}\longrightarrow [0,\infty)$. Given $1<p<\infty$ and $B\subset \tau\,\widehat{B}_0$ we set
$$
\|a\|_{D_p,B}
=
\frac1{a(B)}\sup
\left(\frac1{\mu(B)}\sum_{i} a(B_i)^p\,\mu(B_i)\right)^{1/p},
$$
where the supremum runs over all pairwise disjoint families $\{B_i\}_i$ with $B_i\subset B$.
We say that $a\in D_p(\tau\,\widehat{B}_0)$ provided
$$
\|a\|_{D_p(\tau\,\widehat{B}_0)}:=\sup_{B\subset \tau\,\widehat{B}_0} \|a\|_{D_p,B}<\infty.
$$
Our goal is to prove the following result which, in particular, includes the main result of MacManus and P\'erez \cite[Theorem 1.2]{MaPe}.
\begin{corollary}\label{corol:MP-weak}
Given $0<\rho\le 1$ and  $\tau \geq 1$, assume that for every $B\subset \tau\,\widehat{B}_0$
\begin{equation}\label{genPoinc}
\left(\inf_{c \in \R} \mvint_B |f-c|^\rho \,d\mu \right)^{1/\rho} \leq a(\tau B).
\end{equation}
If $a\in D_p(\tau\,\widehat{B}_0)$, $1<p<\infty$, then
\begin{equation}\label{genPoincImpr}
\|f-f_B\|_{L^{p,\infty}(B)} \lesssim \|a\|_{D_p(\tau\widehat{B}_0)} a(\tau \widehat{B})
\end{equation}
whenever $\tau \widehat{B}\subset \tau \widehat{B}_0$.
\end{corollary}

\begin{proof}
Fix $\tau \widehat{B}\subset \tau \widehat{B}_0$ and note that by Corollary \ref{corol:JNp:X}
$$
\|f-f_{B}\|_{L^{p,\infty}, B}
\lesssim
\|f\|_{JN_{p,\tau}^\rho, \tau \widehat{B}}.
$$
To compute the right hand term let $\{\tau\,B_i\}$ be a pairwise disjoint family so that $\tau\,B_i\subset \tau \widehat{B}$. Then, by \eqref{genPoinc}, we clearly have
\begin{multline*}
\frac1{\mu(\tau \widehat{B})} \sum_i \left( \inf_{c \in \R}
\mvint_{B_i} |f-c|^\rho \,d\mu \right)^{p/\rho}\,\mu(\tau B_i)
\le
\frac1{\mu(\tau \widehat{B})} \sum_i a(\tau\,B_i)^p\,\mu(\tau B_i)
\\
\le
\|a\|_{D_p,\tau \widehat{B}}^p\,  a(\tau \widehat{B})^p
\le
\|a\|_{D_p(\tau \widehat{B}_0)}^p\, a(\tau \widehat{B})^p.
\end{multline*}
Taking the supremum over all such families we easily obtain the desired estimate.
\end{proof}

\begin{remark}
Let us notice that the proof of the previous is an easy consequence of the embedding of the John-Nirenberg spaces into the corresponding weak space along with the definition of the $D_p$ condition. Indeed the same argument yields the following: given a ball $B_0$ and $f\in L^\rho(B_0)$ if \eqref{genPoinc} holds for every ball $B$  such that $\tau\,B\subset B_0$ then $\|a\|_{D_p,B_0}<\infty$ implies that $f\in JN_{p,\tau}^\rho(B_0)$ and moreover
$$
\|f\|_{JN_{p,\tau}^\rho, B_0}
\le
\|a\|_{D_p, B_0}\,a(B_0).
$$
\end{remark}

We can ``optimize'' the previous estimate by ``optimizing'' \eqref{genPoinc}.

\begin{corollary}
Fix a ball $B_0\subset X$ and  $f\in L^{\rho}(B_0)$. For every $B\subset B_0$ set
$$
a_0(B)
:=
\left(\inf_{c \in \R} \mvint_{\tau^{-1}B} |f-c|^\rho \,d\mu \right)^{1/\rho}.
$$
Then, $\|f\|_{JN_{p,\tau}^\rho, B_0}<\infty$ if and only if $\|a_0\|_{D_p,B_0}<\infty$ and in such a case
$$
\|f\|_{JN_{p,\tau}^\rho, B_0}
=
\|a_0\|_{D_p,B_0}\,a_0(B_0).
$$
\end{corollary}

\begin{proof}
Suppose first that $\|f\|_{JN_{p,\tau}^\rho, B_0}<\infty$. Let $\{B_i\}_i$ be a pairwise disjoint family with $B_i\subset B_0$. Then if we write $\widetilde{B}_i=\tau^{-1}\,B_i$ we have that $\{\tau \widetilde{B}_i\}_i$ is a pairwise disjoint family with $\tau \widetilde{B}_i \subset B_0$. Hence, \begin{align*}
\sum_i a_0(B_i)^p\,\mu(B_i)
=
\sum_i \left(\inf_{c \in \R} \mvint_{\widetilde{B}_i} |f-c|^\rho \,d\mu \right)^{p/\rho}\,\mu(\tau \widetilde{B}_i)
\le
\|f\|_{JN_{p,\tau}^\rho, B_0}^p\,\mu(B_0)
\end{align*}
and this readily implies that $\|a_0\|_{D_p,B_0}\,a_0(B_0)\le \|f\|_{JN_{p,\tau}^\rho, B_0}$.

Let us now consider the converse. Assume that $\|a\|_{D_p,B_0}<\infty$. To show that $\|f\|_{JN_{p,\tau}^\rho, B_0}<\infty$ we take $\{\tau B_i\}_i$ pairwise disjoint family with $\tau B_i \subset B_0$. Then,
\begin{align*}
\sum_i \left(\inf_{c \in \R} \mvint_{B_i} |f-c|^\rho \,d\mu \right)^{p/\rho}\mu(\tau B_i)
=
\sum_i a_0(\tau\,B_i)^p\mu(\tau\,B_i)
\le
\|a_0\|_{D_p,B_0}^pa_0(B_0)^p\mu(B_0).
\end{align*}
Taking the supremum over all the possible families we conclude that $\|f\|_{JN_{p,\tau}^\rho, B_0}
\le \|a_0\|_{D_p,B_0}\,a_0(B_0)$.
\end{proof}

\subsection{Weak Gurov-Reshetnyak condition}\label{GRmetric}
The Gurov-Reshetnyak class $GR_\varepsilon(\mu)$, $0<\varepsilon<2$, is defined as the collection of weights
$w \in L^1_{\loc}(X)$ satisfying
$$
\mvint_B |w-w_B|\,d\mu
\le
\varepsilon w_B
$$
for every ball $B\subset X$.
It is known that $w \in GR_\varepsilon(\mu)$ implies
$w \in L^p_{\loc}$ for $1\leq p <p(\varepsilon)$ with $p(\varepsilon)
 \to +\infty$ as $\varepsilon \to 0^+$ \cite[Theorem 3.1]{AB}.
Our approach applies to $GR_\varepsilon(\mu)$, and more generally, its
weak variant
\begin{equation}\label{weakgr}
\mvint_B |w-w_B| \,d\mu \leq \varepsilon w_{\tau B},
\end{equation}
where $\tau \geq 1$ is a fixed parameter. 

As a new result in this metric
setting  we obtain the $L^p$ self-improvement of \eqref{weakgr} for small $\varepsilon$, that is, 
we show that a weak Gurov-Reshetnyak condition implies local higher integrability. 
This, which follows from Theorem \ref{theor:good-lambda:X},  extends the results in \cite{AB} as well as those of T. Iwaniecz \cite{Iw1}, who
studied weak Gurov-Reshetnyak conditions in the Euclidean setting which arise in the study of PDEs.

\begin{theorem}
Fix a ball $B_0$ and let $\tau\ge 1$. Assume that  $0\le w\in L^1(\tau\,\widehat{B}_0)$ satisfies the local weak Gurov-Reshetnyak condition
\begin{equation}\label{eqn:weak-GR}
\mvint_B |w-w_B| \,d\mu \leq \varepsilon w_{\tau B},
\end{equation}
for every $B$ with $\tau\,B\subset \tau \widehat{B}_0$.
If $\varepsilon>0$ is a small enough depending on $\mu$ and $\tau$  then there exists $p(\varepsilon) >1$ \textup{(}see \eqref{ep-p(ep)} below\textup{)} such
that whenever $1 \leq p < p(\varepsilon)$,
we have
\begin{equation}\label{wrh}
\left( \mvint_B |w-w_B|^p \,d\mu \right)^{1/p} \lesssim \varepsilon
w_{\tau \widehat{B}}
\end{equation}
for every $B$ with $\tau\,B\subset \tau \widehat{B}_0$,
and hence $w$ satisfies the following weak reverse H\"older inequality
\begin{equation}\label{wrh-1}
\left( \mvint_B w^p \,d\mu \right)^{1/p}
\lesssim
\mvint_{\tau \widehat{B}} w \,d\mu
\end{equation}
for every $B$ with $\tau\,B\subset \tau \widehat{B}_0$.
Moreover, $p(\varepsilon) \to +\infty$, as $\varepsilon \to 0^+$.
\end{theorem}

\begin{proof}
Fix a ball $B_0'$ satisfying $\tau \widehat{B}_0'\subset \tau \widehat{B}_0$.  We wish to apply Theorem \ref{theor:good-lambda:X} on $B_0'$
to the function $F:=|w-w_{B_0'}|\in L^1(\widehat{B}_0')$. Let $B\in\Bs'=\Bs_{B_0'}$ and note that
$$
F \leq |w-w_{B}|+|w_{B}-w_{B_0'}|=:G^B+H^B.
$$
Condition (ii) on Theorem \ref{theor:good-lambda:X} clearly holds with $\Theta=c_\mu\,\tau^D$. For
(iii) we first observe that $B\in\Bs'$ implies that $\tau\,B\subset\tau \widehat{B}_0'\subset \tau \widehat{B}_0$ (see \eqref{dila-Fam-B}). Hence we can use \eqref{eqn:weak-GR} and obtain
\begin{align*}
\mvint_B G^B \,d\mu & = \mvint_B |w-w_{B}| \,d\mu \leq \varepsilon w_{\tau B}
 \leq \varepsilon \mvint_{\tau B} F \,d\mu +\varepsilon w_{B_0'}.
\end{align*}
Thus, (iii) of Theorem \ref{theor:good-lambda:X} holds with $\delta=\varepsilon$ and
$g^B=\varepsilon\, w_{B_0'}.$ In this case the maximal function
$G_{\Bs'}^*(x)\equiv\varepsilon\, w_{B_0'}$ if $x \in \cup_{\Bs'}B$ and
$G_{\Bs'}^*(x)=0$ otherwise. Therefore, we can invoke Theorem \ref{theor:good-lambda:X} and in particular \eqref{main2} leads to  the desired estimate \eqref{wrh}:
\begin{multline*}
\|w-w_{B_0'}\|_{L^{p},B_0'}
=
\|F\|_{L^{p},B_0'}
\lesssim
\|G_{\Bs'}^*\|_{L^{p},\widehat{B}_0'}
+
F_{\widehat{B}_0'}
\lesssim
\varepsilon\, w_{B_0'}
+
\mvint_{\widehat{B}_0'} |w-w_{\widehat{B}_0'}| \,d\mu
 \lesssim \varepsilon w_{\tau \widehat{B}_0'},
\end{multline*}
where in the last estimate we have used \eqref{eqn:weak-GR} with $\widehat{B}_0'$ in place of $B$ (note that by assumption  $\tau \widehat{B}_0'\subset \tau \widehat{B}_0$). The previous estimate holds if 
$1<p<p(\varepsilon)$ where
\begin{equation}
 p(\varepsilon)=\frac{\log(C_\mu^{-1}\,c_\mu\,\tau^D\,\varepsilon^{-1})}{\log (2\,c_\mu \max\{\tau,3\}^D)}
\qquad\mbox{and}\qquad
0<\varepsilon< \frac1{2\,C_\mu}\, \min\Big\{1, \frac{\tau}{3}\Big\}^D.
\label{ep-p(ep)}
\end{equation}
 Note that $p(\varepsilon)\to +\infty$ as $\varepsilon\to 0^+$.
\end{proof}

\section{Proof of Theorem \ref{theor:good-lambda:X}}\label{section:proof:good-lambda-X}

The proof of Theorem \ref{theor:good-lambda:X} combines ideas from the proof of Theorem \ref{theor:good-lambda:dyadic}
with a Calder\'on-Zygmund type covering in \cite[Lemma 4.4]{MaPe}. We start with two lemmas. In what follows we will use $F_B$ to denote the $\mu$-average of $F$ on $B$.

\begin{lemma}\label{lemma:cz1}
Given $\tau\ge 1$ we set
\begin{equation}\label{lambda0}
\lambda_0:=(15\tau)^D c_\mu\left(1+\frac{1}{\eta}\right)^D.
\end{equation}
Let  $0\le F\in L^1(\widehat{B}_0)$. If $B=B(x_B,r_B) \in \Bs$ \textup{(}cf. \eqref{def-B-cal}\textup{)} is such that $F_B\ge \lambda_0 F_{\widehat{B}_0}$,
then $r_B \leq \frac{\eta}{15\tau}\, r_{B_0}$ and, consequently, $15\tau B \in \Bs$.
\end{lemma}

\begin{proof}
Note that
$$
\lambda_0
\le
\frac{F_B}{F_{\widehat{B}_0}}
\le
\frac{\mu(\widehat{B}_0)}{\mu(B)}
\le
c_\mu\,\left(\frac{r_{\widehat{B}_0}}{r_B}\right)^D
=
c_\mu\,(1+\eta)^D\,\left(\frac{r_{B_0}}{r_B}\right)^D.
$$
This and the definition of $\lambda_0$ gives as desired $r_B \leq \frac{\eta}{15\tau}\, r_{B_0}$ which implies $15\tau B \in \Bs$ by the definition of the family $\Bs$.
\end{proof}

\begin{lemma}\label{cz2}
Let  $0\le F\in L^1(\widehat{B}_0)$ and assume that
$$
\Omega_\lambda := \{x \in \widehat{B}_0: M_{\Bs}F(x) >\lambda\}
\neq\emptyset.
$$
If $\lambda\ge \lambda_0\,F_{\widehat{B}_0}$, where $\lambda_0$ is given in \eqref{lambda0},
there exists a countable family of pairwise disjoint balls
$\{B_i\}_i$ such that
\begin{enumerate}\itemsep=0.3cm
\item[(a)] $\bigcup_i B_i \subset \Omega_\lambda \subset \bigcup_i 5B_i$,

\item[(b)] $15\tau B_i \in \Bs$,

\item[(c)] $F_{B_i} > \lambda$ and

\item[(d)] $F_{\sigma B_i} \leq \lambda$ whenever
$\sigma \geq 2$ and $\sigma B_i  \in \Bs.$
\end{enumerate}
\end{lemma}

\begin{proof}
For every $x \in \Omega_\lambda$, we set
$$
r_x:=\sup \big\{ r_B: \exists\,B=B(x_B,r_B)\in\Bs, \ B\ni x\ \text{and}\ F_B >\lambda\big\}.
$$
By assumption, the set over which the supremum is taken
is non-empty. Moreover, by Lemma \ref{lemma:cz1}, we have $r_x \leq \frac{\eta}{15\tau}r_{B_0}$.
For each $x \in \Omega_\lambda$, we associate a ball $B_x\in\Bs$ with $B_x\ni x$ such
that $F_{B_x} >\lambda$ and $r_x/2 < r_{B_x} \leq r_x$. Applying Vitali's covering theorem,
we get a family of balls $\{B_i\}_i$ with the desired properties.
\end{proof}

\begin{proof}[Proof of Theorem \ref{theor:good-lambda:X}]
Let $K\ge 1$ be a large constant to be chosen and take $
\lambda \geq \lambda_0\, F_{\widehat{B}_0}$.
If $\Omega_\lambda\neq\emptyset$ we can apply the covering lemma to get the family of balls $\{B_i\}_i$ satisfying (a)--(d)
of Lemma \ref{cz2}. We begin by showing that
\begin{equation}\label{th1}
\{x \in 5B_i: M_\Bs F(x) >K\lambda\}=\{x \in 5B_i: M_\Bs(F\chi_{15B_i})(x)>K\lambda\}.
\end{equation}
Since the inclusion $\supset$ is clear, we take an arbitrary $x \in 5B_i$
and assume that $M_\Bs F(x) >K\lambda$. Then there is a ball $B \in \Bs$ with $B\ni x$
such that
\begin{equation}\label{th2}
\mvint_B F \,d\mu >K\lambda.
\end{equation}
We will be done after showing that $B\subset 15\,B_i$ which in turn follows from the fact that $r_B \leq 5\,r_{B_i}$.
In order to obtain the latter suppose otherwise that $r_B > 5\,r_{B_i}$. Then
$B \subset B(x_{B_i},3\,r_B)=\sigma\,B_i$ where $\sigma=3\,r_B/r_{B_i}>15$.
Note that $r_{\sigma\,B_i}=3\,r_B\le \eta\,r_{B_0}$ by \eqref{th2}, the choice of $\lambda$ and Lemma \ref{lemma:cz1}. Hence $\sigma\,B_i \in \Bs$
and we can use Lemma \ref{cz2} and \eqref{doub1} to conclude that
$$
\mvint_B F \,d\mu \leq \frac{\mu(\sigma\,B_i)}{\mu(B)}\,
\mvint_{\sigma\,B_i} F\,d\mu
\leq c_\mu
3^D \lambda \leq K\lambda,
$$
provided $K \geq c_\mu 3^D$. This contradicts \eqref{th2} and hence $r_B \leq 5\,r_{B_i}$. This implies in turn that $B\subset 15\,B_i$ and the proof of \eqref{th1} is complete.

We next consider
$$
E_\lambda := \{x \in \widehat{B}_0: M_{\Bs} F(x) >K\lambda,
G_{\Bs}^*(x) \leq \gamma \lambda\},
$$
where $\gamma\in (0,1)$. Then,
\begin{multline*}
\mu(E_\lambda)
=
\mu(E_\lambda\cap\Omega_{\lambda})
\leq
\sum_i \mu(E_\lambda \cap 5\,B_i)
\\
=
\sum_i
\mu(\{x \in 5\,B_i: M_\Bs(F\chi_{15\,B_i})(x)>K\lambda, G_{\Bs}^*(x) \leq \gamma \lambda\}).
\end{multline*}
In the previous sum it is understood that we keep only the terms where $E_\lambda \cap 5\,B_i\neq\emptyset$. In such a case, we pick $\bar{x}_i\in E_\lambda \cap 5\,B_i$. Recall that by Lemma \ref{cz2} part (b),  $15\,\tau B_i\in\Bs$ and hence $15\,B_i\in\Bs$.
Then using our assumptions (i) and (ii), and Lemma \ref{cz2}, for every $x \in E_\lambda \cap 3B_i$, we have
\begin{multline*}
K\lambda
<
M_\Bs(F\chi_{15B_i})(x)
\leq
M_\Bs(H^{15B_i}\chi_{15B_i})(x)+M_\Bs(G^{15B_i}\chi_{15B_i})(x)
\\
\leq
\Theta\mvint_{15\,\tau\, B_i} F \,d\mu +M_\Bs(G^{15B_i}\chi_{15B_i})(x)
\leq
\Theta\,\lambda +M_\Bs(G^{15B_i}\chi_{15B_i})(x).
\end{multline*}
Taking $K>\Theta$, we use the fact that $M_\Bs$ is of weak type $(1,1)$, assumption (iii) and Lemma \ref{cz2} to obtain \eqref{eqn:good-lambda:X}:
\begin{align*}
\mu(E_\lambda)
&
\leq
\sum_i \mu(\{x \in 5B_i: M_\Bs(G^{15B_i}\chi_{15B_i})(x)>
(K-\Theta)\lambda, G_{\Bs}^*(x) \leq \gamma \lambda\})
\\
&
\lesssim
\sum_i \frac{\mu(15B_i)}{(K-\Theta)\lambda} \mvint_{15\,B_i} G^{15B_i} \,d\mu
\\
& \lesssim
 \sum_i \frac{\mu(B_i)}{(K-\Theta)\lambda} \left(
\delta \mvint_{15\,\tau B_i} F \,d\mu + g^{15B_i}\right),
\\
&
 \le
 \sum_i \frac{\mu(B_i)}{(K-\Theta)\lambda} \left(
\delta \lambda + G_{\Bs}^*(\bar{x}_i)\right)
 \lesssim
\frac{\delta+\gamma}{K-\Theta}
\mu(\Omega_\lambda).
\end{align*}

We next obtain \eqref{main1} and \eqref{main2}. Observe first that \eqref{eqn:good-lambda:X} gives
for every $\lambda \geq \lambda_0 F_{\widehat{B}_0}$:
\begin{multline*}
\mu(\Omega_{K\lambda})
\leq
\mu(E_\lambda)+\mu(\{x \in \widehat{B}_0:G_{\Bs}^*(x) >\gamma \lambda\})
\\
\leq
C_\mu\frac{\delta+\gamma}{K-\Theta} \mu(\Omega_\lambda)
+\mu(\{x \in \widehat{B}_0:G_{\Bs}^*(x) >\gamma \lambda\}),
\end{multline*}
and hence, for all $\lambda >0$,
\begin{align}\label{dee}
\frac{\mu(\Omega_{K\lambda})}{\mu(\widehat{B}_0)}
\leq
C_\mu\frac{\delta+\gamma}{K-\Theta} \frac{\mu(\Omega_\lambda)}{\mu(\widehat{B}_0)}
+\frac{\mu(\{x \in \widehat{B}_0:G_{\Bs}^*(x) >\gamma \lambda\})}{\mu(\widehat{B}_0)}
 +
\chi_{\{0<\lambda< \lambda_0 F_{\widehat{B}_0}\} }(\lambda).
\end{align}
We now proceed as in the dyadic case. Choose $K=2\,\max\{\Theta, c_\mu\,3^D\}$ and
assume that $1<p<\frac{\log(1/(C_\mu\,\delta\, \Theta^{-1}))}{\log K}$. Using \eqref{dee} it follows that
\begin{multline*}
I_N
:
=\sup_{0<\lambda \leq N} \lambda^p \frac{\mu(\Omega_\lambda)}{\mu(
\widehat{B}_0)}
=
K^p\sup_{0<\lambda \leq \tfrac{N}{K}} \lambda^p \frac{\mu(\Omega_{K\lambda})}
{\mu(\widehat{B}_0)}
\\
\leq
K^p\,C_\mu\frac{\delta+\gamma}{\Theta}\,I_N
+\frac{K^p}{\gamma^p}\,\|G_{\Bs}^*\|_{L^{p,\infty}, \widehat{B}_0}^p
+
K^p\,(\lambda_0  F_{\widehat{B}_0})^p.
\end{multline*}
Let us observe  that our choice of $p$ guarantees that 
$K^p\,C_\mu\frac{\delta}{\Theta} <1$ and hence we can take $\gamma$ small enough so that $K^p\,C_\mu\frac{\delta+\gamma}{\Theta}<1$. This and the fact that $I_N\le N^p<\infty$ allow us to absorb the first term in the last estimate to obtain
$$
I_N
\le
C_{\Theta,\delta, \tau,\eta,p}\,\big(\|G_{\Bs}^*\|_{L^{p,\infty}, \widehat{B}_0}^p
+(F_{\widehat{B}_0})^p\big).
$$
Finally, \eqref{main1} follows by letting $N\to\infty$ and using that $F\le M_{\Bs}F$ $\mu$-a.e.~on $B_0$ as observed above.

The proof of \eqref{main2} follows the same ideas already employed in the dyadic case and it is therefore omitted.
\end{proof}

\end{document}